\documentclass[11pt]{article}
\usepackage[a4paper,margin=1.25in]{geometry}
\usepackage[T1]{fontenc}
\usepackage{amsmath,amssymb,amsthm,amsfonts}
\usepackage{graphicx}
\usepackage{enumitem}
\usepackage{algorithmic}
\usepackage[hidelinks]{hyperref}
\usepackage[nameinlink,capitalize,noabbrev]{cleveref}

\usepackage{cite}
\usepackage{epstopdf}
\usepackage{caption}
\ifpdf
\DeclareGraphicsExtensions{.eps,.pdf,.png,.jpg}
\else
\DeclareGraphicsExtensions{.eps}
\fi

\usepackage{enumitem}
\setlist[enumerate]{leftmargin=.5in}
\setlist[itemize]{leftmargin=.5in}

\crefformat{equation}{(#2#1#3)}
\crefrangeformat{equation}{(#3#1#4) to~(#5#2#6)}
\crefmultiformat{equation}{(#2#1#3)}{ and~(#2#1#3)}{, (#2#1#3)}{ and~(#2#1#3)}

\DeclareMathOperator*{\essinf}{ess\,inf}

\usepackage{amsopn}
\newcommand{\ord}[1]{\mathcal{O}\left(#1\right)}

\newtheorem{remark}{Remark}
\newtheorem{definition}{Definition}
\newtheorem{lemma}{Lemma}
\newtheorem{theorem}{Theorem}
\newtheorem{corollary}{Corollary}
\newtheorem{proposition}{Proposition}

\title{Bifurcations in Interior Transmission Eigenvalues:\\Theory and Computation}

\author{
Davide Pradovera\\
\small Stockholm University, Department of Mathematics,\\
\small Roslagsv\"agen 26, 11419 Stockholm, Sweden\\
\small \texttt{davide.pradovera@math.su.se}
\and
Alessandro Borghi\\
\small Technical University Berlin, Institute of Mathematics,\\
\small Stra\ss{}e des 17.~Juni 136, 10623 Berlin, Germany\\
\small \texttt{borghi@tu-berlin.de}
\and
Lukas Pieronek\\
\small \phantom{foobar}Independent researcher\phantom{foobar}\\
\small \texttt{pieronek.lukas@gmail.com}
\and
Andreas Kleefeld\\
\small Forschungszentrum J\"ulich GmbH,\\
\small J\"ulich Supercomputing Centre, Wilhelm-Johnen-Str.,\\
\small 52425 J\"ulich, Germany\\
\small University of Applied Sciences Aachen,\\
\small Faculty of Medical Engineering and Technomathematics,\\
\small Heinrich-Mu\ss{}mann-Str.~1, 52428 J\"ulich, Germany\\
\small \texttt{a.kleefeld@fz-juelich.de}
}

\date{}

\begin{document}

\maketitle

\begin{abstract}
The interior transmission eigenvalue problem (ITP) plays a central role in inverse scattering theory and in the spectral analysis of inhomogeneous media. Despite its smooth dependence on the refractive index at the PDE level, the corresponding spectral map from material parameters to eigenpairs may exhibit non-smooth or bifurcating behavior. In this work, we develop a theoretical framework identifying sufficient conditions for such non-smooth spectral behavior in the ITP on general domains. We further specialize our analysis to some radially symmetric geometries, enabling a more precise characterization of bifurcations in the spectrum. Computationally, we formulate the ITP as a parametric, discrete, nonlinear eigenproblem and use a match-based adaptive contour eigensolver to accurately and efficiently track eigenvalue trajectories under parameter variation. Numerical experiments confirm the theoretical predictions and reveal novel non-smooth spectral effects.
\end{abstract}

\bigskip
\noindent\textbf{Keywords:}
Interior transmission eigenvalues; nonlinear eigenvalue problems; spectral bifurcation; parameter-dependent PDEs; contour-integral methods; numerical eigensolvers.

\bigskip
\noindent\textbf{MSC2020:}
35P30; 65H17; 47J10; 78A46; 37G10.

\section{Introduction}
The \emph{interior transmission eigenvalue problem} (ITP) arises in inverse acoustic scattering theory and mathematical physics. Its importance lies in its role in understanding how time-harmonic acoustic waves interact with inhomogeneous media, particularly for reconstructing properties of objects from scattered wave data in inverse problems \cite{cakoni2008transmission,Cakoni2006}. The ITP first appeared in 1986, when Kirsch studied the denseness property of the far-field operator \cite{kirsch1986denseness}. This was later followed by the work of Colton and Monk \cite{Colto1988a}, who studied it in connection with the linear sampling method. These seminal works established the ITP as a central concept in inverse scattering problems, and interior transmission eigenvalues (the ``solutions'' of the ITP) have since played a crucial role in algorithms for such problems. In particular, ITP eigenvalues encode valuable information about the refractive index of a medium and its geometry.

The ITP is computationally challenging due to the mixed boundary conditions and the coupling between different wave fields. This is easily seen in the strong form of the ITP:
\begin{equation}\label{ITPsetup}
    \begin{cases}
		\Delta v + \kappa^2 v=0=\Delta w + \kappa^2n w \quad &\text{in }D,\\
		v-w=0=\partial_{\nu}v-\partial_{\nu}w \quad &\text{on }\partial D.
	\end{cases}
\end{equation}
Here and throughout, the scatterer $D$ is a given bounded domain whose boundary $\partial D$ has outward-pointing normal $\nu$, and $0<n\not\equiv1$ is the refractive index. We call the nonzero wavenumber $\kappa\in\mathbb{C}_{\neq0}:=\mathbb{C}\setminus\{0\}$ an \emph{interior transmission eigenvalue} (ITE) if there exist nontrivial eigenfunctions $v$ and $w$ solving \cref{ITPsetup}. See \cref{sec:ITP} for more details on the problem and on the function spaces where eigenfunctions are sought.

The discreteness \cite{Colto2007,Colto1989c,Hickm2012,kirsch1986denseness,Rynne1991} and existence \cite{Cakon2010,Paiva2008} of real ITEs have been studied in great detail. However, the problem is not self-adjoint and therefore non-real eigenvalues may exist. For some special geometries such as spherically stratified media, proofs of the existence of non-real ITEs have been developed \cite{colton2010analytical,leung2012complex, xu2017estimates, colton2015distribution,leung2017existence}, although the general case remains open. On the computational side, several numerical methods have been developed to compute (real and non-real) ITEs, e.g., by finite-element methods \cite{Ji2013,Sun2016}, boundary-element methods \cite{Cosso2011,cosso2013,kleefeld2013numerical,zeng2016spectral,cakonikress}, the inside-outside-duality method \cite{peterskleefeld}, and the modified method of fundamental solutions \cite{kleefeldpieronek2018}. See \cite{bookcakoni2022} for a recent comprehensive overview.

% Given its inverse-problem motivation, it is crucial to understand how ITP solutions, i.e., its eigenvalues and eigenfunctions, vary under changes in the ITP parameters, namely, the domain geometry and the refractive index. 

In this work, we investigate the properties of the map $n\mapsto(\kappa,v,w)$ implicitly defined by \cref{ITPsetup}. Although the ITP depends smoothly on the refractive index $n$, the above-defined map \emph{may not be smooth}. For instance, it was shown in \cite{PieKle24} that certain domain geometries (disks in 2D and balls in 3D) may lead to singular, bifurcating behavior in the ITEs as $n$ varies. Understanding such parameter dependence is essential for assessing the stability and sensitivity of inverse scattering algorithms. As a notable example, many sampling methods rely on excluding real transmission eigenvalues from the probing wave numbers. The stability of such approaches may be strongly affected near bifurcation points, where real transmission eigenvalues suddenly become non-real or, conversely, suddenly reappear on the real axis. From a more general viewpoint, bifurcation analysis can help guiding numerical continuation or optimization methods that rely on smooth spectral behavior.
% However, an in-depth theoretical analysis for general ITPs is still missing. This gap is the main focus of our work.

\subsection{Contribution of the paper}
The novelty of our work is threefold. On the one hand, we identify novel sufficient conditions for non-smooth spectral behavior in the ITP on general domains. Our main results on this are \cref{thm:bifurc,cor:bifurc_sufficient,thm:bifurc_general}. We then specialize our theory to obtain a more concrete characterization of spectral bifurcations in two radially symmetric cases: the ITP on the disk (\cref{prop:disk:bifurc}) and on the annulus (\cref{prop:annulus:bifurc,prop:annulus:bifurc_ord}). Lastly, we show that writing the ITP as a parametric, discrete eigenproblem affords us the convenience of using advanced state-of-the-art tools from computational linear algebra to accurately and efficiently compute the (parametric) spectrum of the ITP. Our numerical tests are performed both in the above-mentioned radially symmetric settings, where separation of variables reduces the ITP to a low-dimensional, nonlinear eigenproblem, and on more general inhomogeneous, non-radially symmetric media, where discretizing the ITP usually leads to a large-scale, linear eigenproblem.

Concerning the last point, our approach combines the Beyn contour-integral method \cite{beyn2012integral}, effective for nonlinear but nonparametric eigenproblems, with the match-based adaptive contour eigensolver (MACE) \cite{PraBor24}, which enables tracking the evolution of eigenvalues as the refractive index varies, even in the presence of non-smooth behavior. Through our numerical experiments, we provide evidence of what is, to our knowledge, novel non-smooth spectral behavior in ITE trajectories. Although interesting and novel in their own right, our numerical tests also serve to complement and validate the theoretical results developed in this work, while simultaneously demonstrating the effectiveness of our computational framework.

% \subsection{Outline of the paper}
% In \cref{sec:ITP,sec:eigs:basics}, we introduce the ITP and the framework of parametric eigenproblems, respectively, and present some of their most important properties. In \cref{sec:paraITP} we discuss the potentially non-smooth behavior of parametric ITE trajectories, including our general theoretical contributions. In \cref{sec:disk,sec:annulus} we focus on the ITP for two specific geometries: the disk and the annulus. In both sections, we develop novel theoretical results and discuss the results of numerical simulations, which confirm our theory and reveal novel spectral effects. Finally, \cref{sec:conclusions} provides concluding remarks and an outlook on future research.% The proofs of our more technical theoretical results are presented in the appendices.

\section{The interior transmission eigenvalue problem}\label{sec:ITP}
In this section we introduce our target eigenvalue problems.

Consider a bounded Lipschitz domain $D\subset\mathbb{R}^d$, $d\in\{1,2,3\}$, and a real-valued refractive index
\begin{equation}\label{eq:indexset}
    n\in\mathcal N:=\left\{n\in L^\infty(D):\essinf n>0\right\}\setminus\{1\}.
\end{equation}
The interior transmission eigenvalues (ITEs) of $D$ with refractive index $n$ are defined as all those $\kappa\in\mathbb{C}_{\neq0}$ for which there exist nontrivial complex-valued eigenfunctions
\begin{multline*}
(v,w)\in\big\{(\phi,\psi)\in[L^2(D;\mathbb C)]^2:\phi-\psi\in H^2(D;\mathbb C),\\
\phi-\psi=0=\partial_\nu(\phi-\psi)\text{ on }\partial D\big\}\setminus\{(0,0)\}
\end{multline*}
(using standard notation for Sobolev spaces), such that \cref{ITPsetup} holds \cite{bookcakoni2022}.
% \begin{equation}\label{eq:ITP}
%     \exists (v,w)\not\equiv(0,0)\quad\text{such that}\quad
%     \begin{cases}
%         \Delta v + \kappa^2 v=0=\Delta w + \kappa^2n w \quad &\text{in }D,\\
%         v-w=0=\partial_{\nu}v-\partial_{\nu}w \quad &\text{on }\partial D
%     \end{cases}
% \end{equation}
% with nontrivial $v,w\in L^2(D,\mathbb{C})$ and $(v-w)\in H^2_0(D;\mathbb C)$.
Note that we must exclude the pathological cases $n\equiv1$ and $\kappa=0$. If $n\equiv1$, the problem becomes singular: any $\kappa\in\mathbb{C}$ is an eigenvalue, since one may easily find eigenfunctions $v\equiv w$ with $\Delta v + \kappa^2 v=0$, without imposing any boundary conditions. By a similar reasoning, $\kappa=0$ is always an eigenvalue for any $n$, since we may choose $v\equiv w$ to be any harmonic function on $D$, again without imposing any boundary conditions. This way, noting that the set of harmonic functions on a non-empty $D$ has infinite dimension, we conclude that $0$ belongs to the \emph{continuous spectrum} of the ITP \cite{Kubrusly2012}. This complexity and, simultaneously, triviality justify the exclusion of $\kappa=0$ from our investigations.

% The ITP \cref{eq:ITP} can be recast as a fourth-order elliptic scalar equation in weak form \cite{bookcakoni2022}, whose natural (complex-valued)} eigenfunction spaces are
% \[
% (v,w)\in\left\{(\phi,\psi)\in[L^2(D;\mathbb C)}]^2:\ \phi-\psi\in H^2(D;\mathbb C),\ \phi-\psi=0=\partial_\nu(\phi-\psi)\text{ on }\partial D}\right\},
% \]
% using standard notation for complex-valued} Sobolev spaces.

If additionally $v,w\in H^1(D;\mathbb C)$, which holds, for instance, by elliptic regularity theory for domains $D$ with $C^3$-continuous boundary \cite{lukas2020}, the eigenfunctions also solve the following weak formulation:
\begin{multline}\label{eq:ITPweak}
    \text{find }(\kappa,(v,w))\in\mathbb{C}_{\neq0}\times\mathcal V_{\neq0}\ :\\
    \int_D\left(\nabla v\cdot\nabla\overline{\phi}-\nabla w\cdot\nabla\overline{\psi}\right)=\kappa^2\int_D\left(v\overline{\phi}-nw\overline{\psi}\right)\ \forall(\phi,\psi)\in\mathcal V,
\end{multline}
where $\mathcal V=\left\{(\phi,\psi)\in[H^1(D;\mathbb C)]^2:\phi-\psi=0\text{ on }\partial D\right\}$ and $\mathcal V_{\neq0}=\mathcal V\setminus\{(0,0)\}$. This reformulation of the ITP is convenient in practice and will be the basis for our general study on complex-valued ITE trajectories, see \cref{sec:paraITP}. 
% Since all weights and differential operators appearing in the ITP are real (in the sense that they map real-valued functions to real-valued functions), the ITP spectrum is closed under complex conjugation: eigenpairs with non-real eigenvalues must come in complex-conjugate pairs.} 

\begin{remark}
    One can show that, for Lipschitz domains $D$, the ITP is also equivalent to a linear eigenvalue problem of the form $\mathcal{L}u=\kappa^{-2}u$, where $u\in[H^2(D,\mathbb{C})]^2$ and $\mathcal{L}$ is an $n$-dependent, non-selfadjoint compact operator \cite{bookcakoni2022}. This alternative formulation is useful to deduce spectral properties of the problem: using standard arguments, one can conclude that the ITP spectrum is countable and has no finite accumulation points. This, in combination with the operator's non-selfadjointness, justifies looking for complex ITEs.
\end{remark}

For radially stratified media, which are of interest in \cref{sec:disk,sec:annulus}, existence of \emph{non-real} ITEs was proven under quite general assumptions in \cite{leung2017existence}. For general domains and refractive indices, only the existence of \emph{real} ITEs has been proven so far in \cite{cakoni2009existence, Cakon2010}.% Furthermore, it is known that ITEs may exist only in certain regions of the complex plane \cite{cakoni2010interior}.

As the simplest example, we recall the ITP on the unit disk with homogeneous refractive index, which allows for closed-form expressions of its eigenfunctions and was also the subject of in-depth analysis in \cite{PieKle24}.% Further examples are presented in \cref{sec:annulus}.

\paragraph{Unit disk.} Let $D=B_1\subset\mathbb R^2$ (the unit disk) and substitute $n\equiv p\in\mathcal N$, cf.~\cref{eq:indexset}. Then ITP eigenfunctions can be expressed by separation of variables in polar coordinates, as
\begin{equation}\label{eq:funDisk}
    v(\rho,\theta)=\alpha J_m(\kappa\rho)\phi(m\theta)\;\text{ and }\;
    w(\rho,\theta)=\beta J_m(\sqrt{p}\kappa\rho)\phi(m\theta),\quad\rho\in[0,1],\theta\in[0,2\pi[,
\end{equation}
where $m\in\mathbb{N}_0:=\{0,1,2,\ldots\}$, $(\alpha,\beta)\neq(0,0)$, $\phi$ denotes either sine (if $m\neq0$) or cosine, and $J_m$ is the index-$m$ Bessel function of the first kind. Enforcing the boundary conditions requires
\begin{equation*}
    \begin{pmatrix}
    J_m(\kappa) & J_m(\sqrt{p}\kappa)\\
    J_m'(\kappa) & \sqrt{p}J_m'(\sqrt{p}\kappa)
    \end{pmatrix}\begin{bmatrix}
    \alpha\\-\beta
    \end{bmatrix}=\begin{bmatrix}
    0\\0
    \end{bmatrix},
\end{equation*}
which allows characterizing the point spectrum as
\begin{equation}\label{eq:disk}
    \Sigma_{\textup{disk}}(p)=\bigcup_{m=0}^\infty\left\{\kappa\in\mathbb{C}_{\neq0}:\textup{det}\begin{pmatrix}
    J_m(\kappa) & J_m(\sqrt{p}\kappa)\\
    J_m'(\kappa) & \sqrt{p}J_m'(\sqrt{p}\kappa)
    \end{pmatrix}=0\right\}.
\end{equation}
See \cite{PieKle24} for more details.

\section{Basics on parametric eigenvalue trajectories}\label{sec:eigs:basics}
The main objective of this work is to study how ITEs vary as the refractive index $n$ changes. For convenience we assume that the refractive index is \emph{parametrized}, i.e., that $n=n_p$, with $p\mapsto n_p(\cdot)\in\mathcal N$, cf.~\cref{eq:indexset}, a smooth-enough map whose parameter $p$ is assumed to be a real \emph{scalar}.% In \cref{sec:disk,sec:annulus} we will look at the simplest case of a homogeneous refractive index $n_p\equiv p^2$ introduced in the previous paragraph, whereas in \cref{sec:fem} we will consider a more general inhomogeneous case.}

Parametrizing the refractive index in this way allows characterizing the resulting ITP \cref{eq:ITPweak} as an \emph{infinite-dimensional, parametric (linear) eigenvalue problem}, although discretization turns it finite-dimensional, with a potential loss of linearity, as will soon be shown. We may describe the (discretized or not) parametrized ITP in abstract form as

% With a parametric refractive index, we can categorize the ITP as a \emph{parametric (nonlinear) eigenvalue problem}\footnote{In this section only, we employ the symbols $\kappa$ and $\mathbf{x}$, which are standard in linear algebra, to denote eigenpairs. We use the ITP-specific $\kappa$ and $(v,w)$ elsewhere.} depending on $p$, which we define in abstract form as
\begin{equation}\label{eq:nlevp}
    \forall p\ \text{find }\kappa=\kappa_p\in\mathbb{C}_{\neq0}\ \exists\mathbf{x}=\mathbf{x}_p\in X\setminus\{\mathbf{0}\}\ \text{such that}\ \mathbf{L}(\kappa_p,p)\mathbf{x}_p =\mathbf{0}.
\end{equation}
Above, $(\kappa,p)\mapsto\mathbf{L}(\kappa,p)$ is a function of two variables, taking values in the space $\mathcal L(X)$ of endomorphisms over some Hilbert space $X$, either $\mathcal V\subset[H^1(D;\mathbb C)]^2$ at the continuous level or some Euclidean space $\mathbb C^m$ after discretization, to which eigenvectors $\mathbf{x}$ belong. Throughout our discussion, we assume that $\mathbf{L}$ is holomorphic in $\kappa$ and at least continuous in $p$, which, in our ITP setting, simply requires $n_p$ to depend continuously on $p$.

As a specific instance of this, which will be of great interest to us, separation of variables \cref{eq:funDisk} allows us to study the ITP on the unit disk through discrete \emph{nonlinear} eigenproblems. More explicitly, given a set of $M$ Bessel indices $\{m_1,\ldots,m_M\}\subset\mathbb{N}_0$, we can define a $(2M)\times(2M)$ block-diagonal nonlinear eigenproblem through
\begin{equation*}
    \mathbf{L}:(\kappa,p)\mapsto\begin{pmatrix}
    J_{m_1}(\kappa) & J_{m_1}(\sqrt{p}\kappa) & & &\\
    J_{m_1}'(\kappa) & \sqrt{p}J_{m_1}'(\sqrt{p}\kappa) & & &\\
     & & \ddots & &\\
     & & & J_{m_M}(\kappa) & J_{m_M}(\sqrt{p}\kappa)\\
     & & & J_{m_M}'(\kappa) & \sqrt{p}J_{m_M}'(\sqrt{p}\kappa)
    \end{pmatrix},
\end{equation*}
cf.~\cref{eq:disk}. This is instrumental in studying the ITEs corresponding to the selected Bessel indices.

\begin{definition}[See, e.g., Chapter 2 of \cite{kato}]\label{def:trajectories}
    We call \emph{eigenvalue trajectory} (respectively \emph{eigenpair trajectory}) any map $p\mapsto\kappa_p$ (respectively $p\mapsto(\kappa_p,\mathbf{x}_p)$) implicitly defined in \cref{eq:nlevp}.
\end{definition}

Due to freedom in indexing, such trajectories are not uniquely determined. However, by a perturbative argument valid whenever $\mathbf{L}$ is continuous in $p$, one can always label eigenvalue trajectories in such a way that they are globally continuous with respect to $p$ \cite{kato}.

More generally, if $\mathbf{L}$ depends smoothly on $p$ (i.e., in our setting, if the mapping $p\mapsto n_p$ is smooth), one might expect the eigenpair trajectories to inherit such smoothness. Thanks to a version of the implicit function theorem, one can show that this is true for trajectories involving \emph{simple eigenpairs}, as well as for \emph{semi}simple eigenpairs, although only in a weakened form.

\begin{theorem}[See, e.g., Sections 2 and 5 in \cite{AnChLa93}]\label{thm:simple}
    Let $(\kappa,p)\mapsto\mathbf{L}(\kappa,p)$ be a square-matrix-valued function, holomorphic in $\kappa$ and continuously differentiable $k$ times (resp.~analytic) with respect to $p$ at $(\kappa^\star,p^\star)\in\mathbb{C}\times\mathbb R$, $k\geq 1$. We assume that $\mathbf{L}$ is \emph{non-degenerate}, in the sense that, for all $p$ in a neighborhood of $p^\star$, there exists some $z\in\mathbb{C}$ such that $\det\mathbf{L}(z,p)\neq0$.

    \begin{itemize}
        \item Let $\kappa^\star$ be a \emph{semisimple} eigenvalue\footnote{The provided definition of semisimplicity is in terms of Jordan chains. One may alternatively define it by asking for equal algebraic and geometric multiplicities of eigenvalue $\kappa^\star$, or by requiring $\kappa^\star$ to be a simple pole of the meromorphic function $\kappa\mapsto\mathbf{L}(\kappa,p^\star)^{-1}$ \cite{AnChLa93}.} at $p=p^\star$, i.e., $\det\mathbf{L}(\kappa^\star,p^\star)=0$ and $\partial_\kappa\mathbf{L}(\kappa^\star,p^\star)\mathbf{v}\notin\textup{Im}\mathbf{L}(\kappa^\star,p^\star)$ for all $\mathbf{v}\in\textup{Ker}\mathbf{L}(\kappa^\star,p^\star)\setminus\{0\}$. There exist $M=\textup{dim}(\textup{Ker}\mathbf{L}(\kappa^\star,p^\star))$ distinct continuous eigenpair trajectories such that $\kappa_{p^\star}=\kappa^\star$. All such trajectories are continuously differentiable (once) at $p^\star$.
        \item Let $\kappa^\star$ be a \emph{simple} eigenvalue at $p=p^\star$, i.e., $\kappa^\star$ is a semisimple eigenvalue and $M=\textup{dim}(\textup{Ker}\mathbf{L}(\kappa^\star,p^\star))=1$. The unique continuous eigenpair trajectory such that $\kappa_{p^\star}=\kappa^\star$ is continuously differentiable $k$ times (resp.~analytic) at $p^\star$.
    \end{itemize}
\end{theorem}

\begin{remark}\label{rem:simpleinfinite}
    Under mild conditions, this result may be extended to the general, infinite-dimensional setting. In such cases, however, the notion of ``eigenvalue (semi)simplicity'' must be characterized in terms of (generalized) eigenspace dimensions. See, e.g., \cite[Section 7.3]{kato}.
\end{remark}

\Cref{thm:simple} is useful because, roughly speaking, in most cases of interest (including our target parametric ITP), one can empirically and theoretically verify that ``most'' eigenpairs are (semi)simple for ``most'' parameter values $p$ \cite{kato}. However, even in cases where $\mathbf{L}$ is smooth, it is possible for \emph{exceptional points} to arise, in the following sense.
\begin{definition}[See, e.g., Sections 2.1.1 and 2.6.4 of \cite{kato}]
    Given a continuous eigenpair trajectory $p\mapsto(\kappa_p,\mathbf{x}_p)$, implicitly defined by a problem \cref{eq:nlevp} with $\mathbf{L}$ analytic in $\kappa$ and continuously differentiable in $p$, an exceptional point $p^\star$ is any point where the eigenpair is not continuously differentiable with respect to $p$.
\end{definition}

In general, exceptional points may be categorized using one of the following labels.
\begin{itemize}
\item \emph{Smooth crossings.} These are points where two or more smooth eigenvalue trajectories cross, leading to semisimple eigenpairs, without any loss of smoothness. These points are exceptional because lack of smoothness may arise from a mislabeling of the eigenvalue curves, although the singularity is ``removable'' through proper labeling. See, e.g., \cref{fig:nonsmoothness} (left).% In some sense, crossings are ``removable'' exceptional points. For instance \cite{rellich},
% \begin{equation*}
%     \mathbf{L}(\lambda,p)=\begin{pmatrix}
%         \lambda-p & 0\\ 0 & \lambda+p
%     \end{pmatrix}
% \end{equation*}
% admits the two eigenpair trajectories
% \begin{equation*}
%     (\lambda_p^{(1)},\mathbf{x}_p^{(1)})=(p,[1,0]^\top)\quad\text{and}\quad(\lambda_p^{(2)},\mathbf{x}_p^{(2)})=(-p,[0,1]^\top),
% \end{equation*}
% cf.~\cref{fig:nonsmoothness} (left). A crossing point emerges at $\lambda=0$ for $p=0$, although eigenpair trajectories are smooth. One may re-label the two eigenpair trajectories so that they are non-smooth:
% \begin{equation*}
%     (\lambda_p^{(1)},\mathbf{x}_p^{(1)})=(|p|,[H(p),1-H(p)]^\top)\quad\text{and}\quad(\lambda_p^{(2)},\mathbf{x}_p^{(2)})=(-|p|,[1-H(p),H(p)]^\top),
% \end{equation*}
% with $H$ the Heaviside function, returning $0$ for negative inputs and $1$ otherwise.

\begin{figure}
    \centering
    \includegraphics{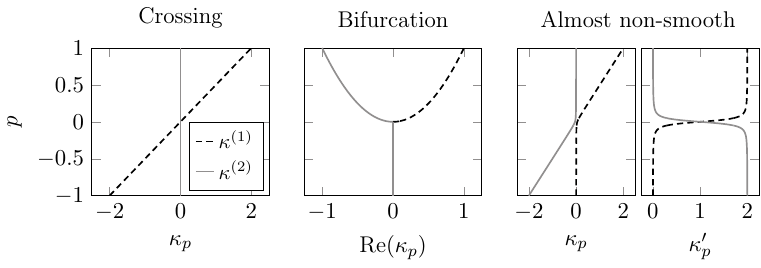}
    \caption{Examples of non-smooth eigenpair behavior. We set $\varepsilon=0.05$ in the right plots.}
    \label{fig:nonsmoothness}
\end{figure}

\item \emph{Algebraic bifurcations.} These are locations where an eigenvalue trajectory can be locally expressed as a function with an algebraic singularity.
\begin{definition}[See, e.g., Section 2.1.1 of \cite{kato}]\label{def:bifurcation}
    An eigenpair trajectory $p\mapsto\kappa_p$ has an (algebraic) bifurcation point of order $M>1$ at $p^\star$ if there exist $(M-1)$ further \emph{distinct} ITE trajectories $p\mapsto(\kappa_p^{(1)},\ldots,\kappa_p^{(M-1)})\in\mathbb C^{M-1}$ with $\kappa_{p^\star}^{(j)}=\kappa_{p^\star}$ for all $j$, such that $(\kappa_p-\kappa_{p^\star})\sim(p-p^\star)^{1/M}\sim(\kappa_p^{(j)}-\kappa_{p^\star})$ as $p\to p^\star$ for all $j$.
\end{definition}

For instance \cite{kato},
\begin{equation*}
    \mathbf{L}(\kappa,p)=\begin{pmatrix}
        \kappa & 1\\ p & \kappa
    \end{pmatrix}
\end{equation*}
admits the two non-smooth eigenpair trajectories
\begin{equation*}
    (\kappa_p^{(1)},\mathbf{x}_p^{(1)})=(p^{1/2},[1,-p^{1/2}]^\top)\quad\text{and}\quad(\kappa_p^{(2)},\mathbf{x}_p^{(2)})=(-p^{1/2},[1,p^{1/2}]^\top),
\end{equation*}
with $p^{1/2}$ being a branch of the complex square root for $p<0$. See \cref{fig:nonsmoothness} (center).% No re-labeling of the eigenpairs allows recovering smoothness at $p=0$, so a quadratic bifurcation point is present at $\lambda=0$ for $p=0$. Note that $M$ distinct eigenvalue trajectories must be present near an order-$M$ bifurcation, e.g., by the fundamental theorem of algebra.

\item \emph{Fundamentally non-smooth behavior.} More complex behavior may arise if non-simplici\-ty persists over positive-measure ranges of $p$, in case of \emph{permanent degeneracy}, or if the dependence on $p$ is less regular \cite[Sections 2.1.1 and 2.5.5]{kato}.% At crossings and bifurcations, eigenvalues are locally non-simple only at a single point. More complex behavior may arise if non-simplicity persists over positive-measure ranges of $p$. This is the case, e.g., for singular NEPs such that $\det\mathbf{L}\equiv0$, which admit any $\lambda\in\mathbb{C}$ as eigenvalue. A different exceptional behavior is displayed by \emph{permanently degenerate} NEPs \cite{kato}, which are non-singular but whose eigenpairs remain defective for most values of $p$, e.g.,
% \begin{equation*}
%     \mathbf{L}(\lambda,p)=\begin{pmatrix}
%         \lambda & p\\ 0 & \lambda
%     \end{pmatrix},
% \end{equation*}
% whose unique eigenvalue $\lambda_p\equiv0$ has eigenvector $\mathbf{x}_p\equiv[1,0]^\top$ at all $p$ \emph{except} at $p=0$, where any vector is an eigenvector.
\end{itemize}

It is also useful to note that eigenpair trajectories, even if smoothly dependent on $p$, may be highly sensitive to $p$ due to an exceptional point arising for a nearby \emph{complex} value of the parameter. For instance, based on an example from \cite{stephen},
% In some cases, we may talk of ``almost non-smooth behavior'' if the eigenpair trajectories, while smoothly dependent on the parameters, are highly sensitive to it, e.g., because their derivatives with respect to $p$ are large. As an example, fix $\varepsilon\neq0$ and consider, based on an example from \cite{stephen},
\begin{equation*}
    \mathbf{L}(\kappa,p)=\begin{pmatrix}
        \kappa-2p & \varepsilon\\ \varepsilon & \kappa
    \end{pmatrix},
\end{equation*}
admits the two smooth eigenpair trajectories
\begin{equation*}
    (\kappa_p^{(1,2)},\mathbf{x}_p^{(1,2)})=(p\pm\sqrt{p^2+\varepsilon^2},[\varepsilon,p\mp\sqrt{p^2+\varepsilon^2}]^\top),
\end{equation*}
cf.~\cref{fig:nonsmoothness} (right). %The eigenpair trajectories have a large $\ord{\varepsilon^{-1}}$ second derivative for $p$ near $0$. 
This labeling of the eigenvalues, which results in smooth trajectories for $0\neq\varepsilon\to0$, yields non-smooth trajectories at the limit $\varepsilon=0$. We will informally refer to such cases as ``almost-exceptional points''.

\section{Parametric ITE trajectories}\label{sec:paraITP}
Looking specifically at ITE trajectories, in \cite{PieKle24} the authors provide a characterization of a family of exceptional points for the ITP on the unit disk with homogeneous refractive index. As summarized below in \cref{prop:disk:bifurc}, it was shown that this problem displays bifurcation points located on the real axis, whenever the imaginary part of an ITE trajectory becomes zero, i.e., when a non-real eigenvalue trajectory intersects the real axis.

As our first novel result, we show that this result applies also in arbitrary domains, with refractive indices that depend on the parameter in general ways. In particular, the framework also includes inhomogeneous media, since the refractive index may be non-uniform or even spatially discontinuous.

\begin{theorem}\label{thm:bifurc}
    Consider a generic ITP in weak form \cref{eq:ITPweak}. Assume that the refractive index is dependent on a parameter $n=n_p\in\mathcal N$, cf.~\cref{eq:indexset}, with $p\mapsto n_p$ continuously differentiable over $(p^\star-\varepsilon,p^\star+\varepsilon)$, for some fixed $p^\star\in\mathbb R$ and some small enough $\varepsilon>0$.
    
    Take an eigenpair trajectory
    \begin{equation}\label{eq:eigenpair}
        (p^\star-\varepsilon,p^\star+\varepsilon)\ni p\mapsto(\kappa_p,(v_p,w_p))\in\mathbb{C}_{\neq0}\times\mathcal V_{\neq0},
    \end{equation}
    continuous (in $\mathbb{C}\times[H^1(D;\mathbb C)]^2$) on $(p^\star-\varepsilon,p^\star+\varepsilon)$ and continuously differentiable (in $\mathbb{C}\times[H^1(D;\mathbb C)]^2$) on $(p^\star-\varepsilon,p^\star+\varepsilon)\setminus\{p^\star\}$. Also, assume that $\int_Dn_{p^\star}'\left|w_{p^\star}\right|^2\neq 0$, with $n_{p^\star}'\in L^\infty(D)$ denoting the derivative of $p\mapsto n_p$ at $p=p^\star$.

    Assume that $\kappa_p\in\mathbb{C}\setminus\mathbb{R}$ for all $p\in(p^\star-\varepsilon,p^\star)$ (resp., for all $p\in(p^\star,p^\star+\varepsilon)$) but $\kappa_{p^\star}\in\mathbb{R}\setminus\{0\}$, i.e., the chosen eigenvalue trajectory ``touches'' the real axis from outside it when $p=p^\star$. Then $p=p^\star$ gives rise to an exceptional point, since the eigenpair trajectory is non-differentiable there.
\end{theorem}

The above result shows that any smooth ITP (not just the homogeneous one on the disk) may display non-smooth behavior in its eigenpair trajectories whenever the imaginary part of an eigenvalue trajectory transitions from nonzero to zero. The proof of \cref{thm:bifurc} relies on the following two technical lemmas.

\begin{lemma}\label{lem:zeronorm}
    Consider the ITP in weak form \cref{eq:ITPweak}, with $n\in\mathcal N$, cf.~\cref{eq:indexset}. Let $\kappa\in\mathbb{C}\setminus\mathbb{R}$ be a \emph{non-real} eigenvalue with eigenfunction pair $(v,w)\in\mathcal V_{\neq0}$. Then $\|v\|_{L^2(D;\mathbb C)}=\|\sqrt nw\|_{L^2(D;\mathbb C)}$.
\end{lemma}
\begin{proof}
    See \cite[Lemma 2.1]{PieKle24}. For an alternative proof, it suffices to plug the eigenfunctions as test functions ($\phi=v$ and $\psi=w$) in the weak formulation \cref{eq:ITPweak} and take the imaginary part of the resulting expression.
\end{proof}

\begin{lemma}\label{lem:bifurc_ord}
    Consider an eigenpair trajectory $p\mapsto(\kappa_p,(v_p,w_p))$ for an ITP in weak form \cref{eq:ITPweak} whose refractive-index parametrization is continuously differentiable. Define
    \begin{equation}\label{eq:bifurc_indicator}
        I(p):=\int_D\left(\left|v_p\right|^2-n_p\left|w_p\right|^2\right)=\|v_p\|_{L^2(D;\mathbb C)}^2-\|\sqrt{n_p}w_p\|_{L^2(D;\mathbb C)}^2
    \end{equation}
    and assume that the eigenpair trajectory is continuously differentiable at $p$ and such that $\kappa_p\in\mathbb R\setminus\{0\}$. Then the derivative of the eigenvalue trajectory at $p$ satisfies
    \begin{equation}\label{eq:bifurc_ord}
        I(p)\kappa_p'=\frac12\kappa_p\int_Dn_p'\left|w_p\right|^2,
    \end{equation}
    with $n_p'$ denoting the derivative of the refractive-index map $p\mapsto n_p$.
\end{lemma}
\begin{proof}
    We can differentiate the weak formulation \cref{eq:ITPweak} with respect to $p$ (for fixed test functions $\phi$ and $\psi$), e.g., in a Fr\'echet sense, to obtain a weak formulation satisfied by the eigenpair derivatives $(\kappa_p',(v_p',w_p'))\in\mathbb C\times\mathcal V$:
    \begin{equation*}
        \int_D\left(\nabla v_p'\cdot\nabla\overline{\phi}-\nabla w_p'\cdot\nabla\overline{\psi}\right)=2\kappa_p\kappa_p'\int_D\left(v_p\overline{\phi}-n_pw_p\overline{\psi}\right)+\kappa_p^2\int_D\left(v_p'\overline{\phi}-n_p'w_p\overline{\psi}-n_pw_p'\overline{\psi}\right).
    \end{equation*}
    
    Plugging the eigenfunctions as test functions $\phi=v_p$ and $\psi=w_p$ yields
    \begin{equation*}
        \int_D\left(\nabla v_p'\cdot\nabla\overline{v_p}-\nabla w_p'\cdot\nabla\overline{w_p}\right)=2\kappa_p\kappa_p'I(p)+\kappa_p^2\int_D\left(v_p'\overline{v_p}-n_p'\left|w_p\right|^2-n_pw_p'\overline{w_p}\right).
    \end{equation*}
    By the complex conjugate of the weak formulation \cref{eq:ITPweak} with test functions $\phi=v_p'$ and $\psi=w_p'$, the left-hand side equals
    \begin{equation*}
        \overline{\kappa_p}^2\int_D\left(v_p'\overline{v_p}-n_pw_p'\overline{w_p}\right),
    \end{equation*}
    so that
    \begin{equation*}
        2\kappa_p\kappa_p'I(p)-\kappa_p^2\int_Dn_p'\left|w_p\right|^2=(\overline{\kappa_p}^2-\kappa_p^2)\int_D\left(v_p'\overline{v_p}-n_pw_p'\overline{w_p}\right).
    \end{equation*}
    Since $\kappa_p\in\mathbb R\setminus\{0\}$ by assumption, the right-hand side vanishes and the claim follows.
\end{proof}

We are finally ready to prove \cref{thm:bifurc}.

\begin{proof}[Proof of \cref{thm:bifurc}]
    By way of contradiction, assume that $p\mapsto(\kappa_p,(v_p,w_p))$ is continuously differentiable at $p=p^\star$, so that \cref{lem:bifurc_ord} holds there. Recall that $\kappa_p\in\mathbb C\setminus\mathbb R$ for all $p$ in a one-sided neighborhood of $p^\star$, so that $I(p)=0$ for all such values of $p$ by \cref{lem:zeronorm}. Since the eigenfunction trajectories are continuous, so is $I$. As such, $I(p^\star)=0$ and \cref{eq:bifurc_ord} cannot hold at $p=p^\star$, a contradiction.
\end{proof}

Before proceeding further, we note that \cref{thm:bifurc} makes the technical assumption that $\int_Dn_{p^\star}'\left|w_{p^\star}\right|^2\neq 0$, which is necessary to ensure that \cref{eq:bifurc_ord} yields a contradiction. The following practical condition is sufficient to guarantee that this property holds.
\begin{lemma}\label{lem:nprime}
    With the notation of \cref{thm:bifurc}, if $n_{p^\star}'>0$ a.e.~on $D$ or $n_{p^\star}'<0$ a.e.~on $D$, then $\int_Dn_{p^\star}'\left|w_{p^\star}\right|^2\neq 0$.
\end{lemma}
\begin{proof}
    By the assumption on the sign of $n_{p^\star}'$, $\int_Dn_{p^\star}'\left|w_{p^\star}\right|^2=0$ requires $w_{p^\star}=0$ a.e.~on $D$, i.e., since $w_{p^\star}\in H^1(D)$, $w_{p^\star}\equiv0$. But then the ITP equations \cref{ITPsetup} imply that $v_{p^\star}$ satisfies both homogeneous Dirichlet \emph{and} homogeneous Neumann boundary conditions on $\partial D$. Unique-continuation arguments \cite[Section 4.2.2]{kirsch} imply that both eigenfunctions are then identically zero, a contradiction.
\end{proof}

Intuitively, the above result states that, if $n_p'$ has a unique nonzero sign over $D$, variations in $p$ have a nonzero ``net effect'' over the domain, which breaks spectral smoothness. It is our belief that it might be possible to manufacture ITPs where $n_{p^\star}'$ changes sign within $D$ and exceptional points disappear because of it. However, further research is needed to verify this. Indeed, there is a substantial gap in the literature for such ``zero-net-effect'' cases, e.g., ITPs whose inhomogeneous refractive indices are somewhere below and somewhere above unity \cite{bookcakoni2022}.

\subsection{Characterizing exceptional points by a scalar indicator}

By inspection of the proof of \cref{thm:bifurc}, we can characterize when exceptional phenomena occur in an alternative way.

\begin{corollary}\label{cor:bifurc}
    \Cref{thm:bifurc} remains valid if the hypothesis that ``$\kappa_p\in\mathbb{C}\setminus\mathbb{R}$ for all $p\in(p^\star-\varepsilon,p^\star)$ (resp., for all $p\in(p^\star,p^\star+\varepsilon)$)'' is replaced by ``$I(p^\star)=0$, cf.~\cref{eq:bifurc_indicator}''.
\end{corollary}
\begin{proof}
    The non-realness of the ITE trajectory on one side of $p^\star$ is used in the proof of \cref{thm:bifurc} only to conclude that $I(p^\star)=0$, a fact that we are explicitly assuming here.
\end{proof}

This allows us to conclude that non-smooth eigenpair behavior generally arises whenever the key condition $\|v_{p^\star}\|_{L^2(D;\mathbb C)}\neq\|\sqrt{n_{p^\star}}w_{p^\star}\|_{L^2(D;\mathbb C)}$ is violated at some \emph{real} eigenvalue $\kappa_{p^\star}$ (since $I(p)\equiv0$ on non-real ITE trajectories by \cref{lem:zeronorm}). Crucially, the hypotheses of \cref{cor:bifurc} do not require non-realness of the eigenvalues, so that the result applies even to purely real eigenvalue trajectories. Unlike \cref{thm:bifurc}, this enables the identification of non-smoothness in the ITP eigenpair curves even without leaving the real axis. Specifically, given an eigenpair trajectory whose eigenvalue is real for all $p$, we can easily determine whether the trajectory is involved in any exceptional points by just checking if the indicator\footnote{Even though $I$ depends on the eigenpair trajectory, we denote it as a function of just $p$ for conciseness.} $I$ from \cref{eq:bifurc_indicator} is zero at any $p$.

\subsection{When exceptional points are bifurcations}

Recall \cref{lem:bifurc_ord}, which only applies to smooth portions of eigenpair trajectories. We may use it in a limiting argument to study exceptional points. Notably, the following result provides conditions under which the \emph{exceptional points} in eigen\emph{pair} trajectories predicted by \cref{thm:bifurc,cor:bifurc} correspond to \emph{algebraic bifurcation points} in eigen\emph{value} trajectories.

\begin{theorem}\label{cor:bifurc_sufficient}
    Consider an ITP as in \cref{thm:bifurc}, whose refractive index $n_p$ depends analytically on $p\in(p^\star-\varepsilon,p^\star+\varepsilon)$ and such that $\int_Dn_{p^\star}'\left|w_{p^\star}\right|^2\neq 0$. Take a continuous eigenpair trajectory \cref{eq:eigenpair} that is continuously differentiable on a punctured neighborhood of $p^\star$, such that $\kappa_{p^\star}\in\mathbb R$. Let $M\in\{2,3,\ldots\}$. The following statements are equivalent:
    \begin{enumerate}[label=(\roman*)]
        \item\label{itemtwo} $p^\star$ is an order-$(1-\frac1M)$ root of $I$, i.e., $I(p)=(p-p^\star)^{1-1/M}\widetilde{I}(p)$ for some function $\widetilde{I}$ that is continuous and nonzero on a one-sided closed neighborhood of $p^\star$;
        \item\label{itemone} $I(p)=(\kappa_p-\kappa_{p^\star})^{M-1}\overline{I}(p)$ for some function $\overline{I}$, continuous and nonzero on a one-sided closed neighborhood of $p^\star$;
        \item\label{itemthree} the eigenpair trajectory has a bifurcation of order $M$ at $p^\star$, cf.~\cref{def:bifurcation}.
    \end{enumerate}
\end{theorem}
\begin{proof}
    First, note that, under either \ref{itemtwo} or \ref{itemone}, $p^\star$ must be an exceptional point. Indeed, since $I$ is non-identically zero on a one-sided neighborhood $U_{p^\star}$ of $p^\star$, the eigenvalue trajectory must be purely real for $p\in U_{p^\star}$, cf.~\cref{lem:zeronorm}. We can thus apply \cref{lem:bifurc_ord}. By continuity of the eigenpair trajectory, we can look at the limit of the right-hand side of \cref{eq:bifurc_ord} as $U_{p^\star}\ni p\to p^\star$. Since both $\kappa_{p^\star}$ and $\int_Dn_{p^\star}'\left|w_{p^\star}\right|^2$ are nonzero by assumption while $I(p^\star)=0$, $\kappa_p'\sim1/I(p)$ is unbounded as $U_{p^\star}\ni p\to p^\star$.
    
    To determine the local behavior of $\kappa_p$ around the singularity (which will be later needed to characterize a bifurcation and its order), we can use a perturbation argument as in \cite{HrLa99}. Namely, thanks to the analytic dependence on $p$, a \emph{Puiseux series} for $\kappa_p$ must exist: for some $M'\in\{2,3,\ldots\}$, and $c\in\mathbb{C}_{\neq0}$,
    \begin{equation*}
        \kappa_p=\kappa_{p^\star}+c(p-p^\star)^{1/M'}+\ord{(p-p^\star)^{2/M'}}\quad\text{for }p\in U_{p^\star},
    \end{equation*}
    with $(\cdot)^{1/M'}$ a fixed branch of the complex $M'$-th root.
    
    We move now to the equivalence of the three statements.
    \begin{itemize}
        \item $\ref{itemtwo}\Rightarrow\ref{itemone}$. We restrict $p$ to $U_{p^\star}$. Note that $\kappa_p'=\ord{(p-p^\star)^{-1+1/M'}}$ by differentiating the Puiseux series. Moreover, $\kappa_p'=\ord{1/I(p)}=\ord{(p-p^\star)^{-1+1/M}}$ by \cref{eq:bifurc_ord}. This shows that $M=M'$ by equating the asymptotic rates of $\kappa_p'$ as $U_{p^\star}\ni p\to p^\star$ in the Puiseux series and in \cref{eq:bifurc_ord}. Moreover,
        \begin{align*}
            I(p)=(p-p^\star)^{1-1/M}\widetilde{I}(p)&=\left(c^{M-1}(\kappa_p-\kappa_{p^\star})^{M-1}+\ord{(p-p^\star)^{2-2/M}}\right)\widetilde{I}(p)\\
            &=:(\kappa_p-\kappa_{p^\star})^{M-1}\overline{I}(p).
        \end{align*}
        \item $\ref{itemone}\Rightarrow\ref{itemthree}$. As above, we restrict $p$ to $U_{p^\star}$ and we show that $M=M'$ by equating the asymptotic rates of $\kappa_p'$ as $U_{p^\star}\ni p\to p^\star$ in the Puiseux series and in \cref{eq:bifurc_ord}. An order-$(1/M)$ algebraic singularity in the trajectory of one of the roots of an analytic function with analytic dependence on a parameter characterizes an order-$M$ bifurcation \cite{HrLa99}.
        \item $\ref{itemthree}\Rightarrow\ref{itemtwo}$. We first need to show that any algebraic bifurcation around a real eigenvalue must involve at least one locally purely real eigenvalue curve, so that we may use \cref{lem:bifurc_ord} to conclude that $I$ is not uniformly zero. This is easily seen by going back to the above Puiseux series. Extending it to all $M$ trajectories involved in the bifurcations results in
        \begin{equation*}
            \kappa_p^{(j)}=\kappa_{p^\star}+ce^{2ij\pi/M}(p-p^\star)^{1/M}+\ord{|p-p^\star|^{2/M}},\quad j=1,\ldots,M,
        \end{equation*}
        with $c\neq0$ by assumption. Since the ITP spectrum is closed under complex conjugation \cite{PieKle24}, the complex argument of $c$ is constrained: $c=|c|e^{ik\pi/M}$ for some $k\in\mathbb{Z}$. As such, there must exist at least one real eigenvalue curve for $p<p^\star$ (if $k$ is odd) or for $p>p^\star$ (if $k$ is even). The claim then follows, as in the previous case, by equating the asymptotic rates of $\kappa_p'$ as $U_{p^\star}\ni p\to p^\star$ in the Puiseux series and in \cref{eq:bifurc_ord}.%\footnote{Specifically, if $M$ is odd, there must exist exactly one real curve for $p<p^\star$ and exactly one real curve for $p>p^\star$. On the other hand, if $M$ is even, there must exist exactly two real curves for $p$ on one side of $p^\star$ and none on the other side.}
    \end{itemize}
    The claim follows.
\end{proof}

\Cref{cor:bifurc_sufficient} is useful for studying the existence of (real) exceptional points in ITE trajectories. We showcase this in our numerical experiments below. As a related point of note, we can also observe the following.

\begin{remark}\label{rem:angle}
    The proof of \cref{cor:bifurc_sufficient} exploits the fact that the angles that the bifurcating eigenvalue trajectories form with respect to the real axis are constrained. In particular, an order-$M$ bifurcation must involve a (locally) purely real branch in the eigenvalue trajectories:
    \begin{itemize}
        \item If $M$ is even, there must exist exactly two real branches for $p<p^\star$ (resp.~$p>p^\star$) and no real branches for $p>p^\star$ (resp.~$p<p^\star$).
        \item If $M$ is odd, there must exist exactly one real branch for $p\neq p^\star$.
    \end{itemize}
\end{remark}

Differently from our previous results, \cref{cor:bifurc_sufficient} requires analytic dependence of $n_p$ on $p$, rather than just continuous differentiability, in order to invoke tools from analytic perturbation theory \cite{HrLa99,kato}. This is necessary since bifurcations are, by definition, phenomena that require locally analytic parameter dependence. We note that a smooth (in fact, polynomial) parametrization $p\mapsto n_p$ is already natural in many settings. On the other hand, we stress that no smoothness in the spatial dependence of $n_p$ is required.

\subsection{Global result on spectral non-smoothness for ITPs}

To summarize our results, we state the following theorem that foregoes assumptions on the smoothness of ITE trajectories in favor of smoothness requirements on the ITP, which are more easily verified.

\begin{theorem}\label{thm:bifurc_general}
    Let $\mathcal P=(p_{\min},p_{\max})\subset\mathbb R$ and consider a generic ITP over a $C^1$ domain $D$, whose parametric refractive index $\mathcal P\ni p\mapsto n_p\in\mathcal N$ admits a continuous derivative (in $L^\infty(D)$) with respect to $p\in\mathcal P$, which satisfies either $n_p'>0$ a.e.~on $D$ or $n_p'<0$ a.e.~on $D$ for all $p\in\mathcal P$. Take a continuous eigenpair trajectory, whose nonzero eigenvalue is semisimple, cf.~\cref{thm:simple}, uniformly in $p$, with the possible exception of a finite set of $p$ values.
    
    Then the eigenpair trajectory is continuously differentiable, with the possible exception of the values of $p\in\mathcal P$ where the eigenvalue ceases to be semisimple. Moreover, \cref{thm:bifurc}, \cref{lem:bifurc_ord}, and \cref{cor:bifurc} apply. If $n_p$ depends analytically on $p\in\mathcal P$, then \cref{cor:bifurc_sufficient} also applies.
\end{theorem}
\begin{proof}
    Due to the existence of eigenvalue-free regions \cite{cakoni2010interior}, the ITP is obviously non-de\-ge\-ne\-rate. \Cref{thm:simple} can then be applied, since its smoothness assumptions are verified, cf.~\cref{rem:simpleinfinite}. As such, any (portion of an) eigenpair trajectory such that the eigenvalue remains semisimple is continuously differentiable in $p$. The hypothesis that $\int_Dn_{p^\star}'\left|w_{p^\star}\right|^2\neq 0$, needed to apply \cref{thm:bifurc}, \cref{cor:bifurc}, and \cref{cor:bifurc_sufficient}, holds by \cref{lem:nprime}.
\end{proof}

%This shows under which conditions the results from \cite{PieKle24} may be applied in vastly more general settings.

\section{The ITP for the disk}\label{sec:disk}
Recall the ITP on the unit disk with homogeneous refractive index, whose spectrum is characterized by the family of $2\times2$ nonlinear eigenvalue problems in \cref{eq:disk}. As already mentioned, our results apply to this problem.

\begin{proposition}\label{prop:disk:bifurc}
    Consider the ITP for the disk with parametric refractive index $n_p\equiv p$, with $\mathcal P=(p_{\min},p_{\max})\subset\mathbb R_{>0}\setminus\{1\}$, and take any continuous eigenpair trajectory whose nonzero eigenvalue is uniformly (in $p$) semisimple except at a finite set of values of $p$.
    
    Then the eigenpair trajectory is locally analytic, with the possible exception of the values of $p$ where the eigenvalue ceases to be semisimple. Moreover, \cref{thm:bifurc}, \cref{lem:bifurc_ord}, \cref{cor:bifurc}, and \cref{cor:bifurc_sufficient} apply. Notably, any bifurcation involving a locally real eigenvalue occurs precisely at a Bessel zero and is cubic.
\end{proposition}
\begin{proof}
    The refractive index $n_p$ depends analytically on $p$ and $n_{p^\star}'\equiv1>0$, so that \cref{thm:bifurc_general} implies all claims except the one on bifurcation order.
    
    For this last claim, we rely on \cref{rem:angle} to conclude that, near a real exceptional point, there must exist at least one eigenvalue trajectory that is purely real for $p$ on a one-sided neighborhood of the exceptional point $p^\star$. We look at the indicator $I$ on such trajectory. As shown in \cite[Lemma 2.2]{PieKle24}, for a real eigenvalue trajectory,
    \begin{equation*}
        I(p)=\frac{a_m\pi}2\left(1-n_p\right)J_m(\kappa_p)^2,
    \end{equation*}
    with $m\in\mathbb{N}_0$ being the Bessel order, cf.~\cref{eq:funDisk}, and $a_m=2$ if $m=0$ and $a_m=1$ otherwise. At (nonzero) roots $p^\star$ of $I$, one must have $J_m(\kappa_{p^\star})=0$ and
    \begin{equation*}
        I(p)=\frac{a_m\pi}2\left(1-n_p\right)J_m'(\kappa_{p^\star})^2(\kappa_p-\kappa_{p^\star})^2+\ord{(\kappa_p-\kappa_{p^\star})^3},
    \end{equation*}
    with $J_m'(\kappa_{p^\star})\neq0$ since nontrivial Bessel zeros are simple. \Cref{cor:bifurc_sufficient} yields the claim.
\end{proof}

In this way, we have recovered all claims from \cite[Corollary 2.3, Lemma 2.4, Lemma 2.6, and Theorem 2.12]{PieKle24} using our novel general results.

\subsection{Numerical tests on the unit disk}\label{sec:disk:num}

To numerically validate our results, we adopt the algorithm introduced in \cite{PraBor24}, designed to approximate solutions of general parametric nonlinear eigenvalue problems. We refer to this algorithm as ``match-based adaptive contour eigensolver'' (MACE). MACE is designed to approximate eigenvalue trajectories over a parameter range $p\in[p_{\min},p_{\max}]\subset\mathbb R$. This is achieved through a collocation approach, by solving the eigenvalue problem at some values of $p$ and then interpolating the (post-processed) eigenvalues to form continuous trajectories. An adaptive loop is employed to more efficiently explore the parameter range. MACE also includes tailor-made strategies for higher-order approximation of bifurcating trajectories. This makes it particularly attractive in our setting.

\begin{figure}[t]
    \centering
    \includegraphics{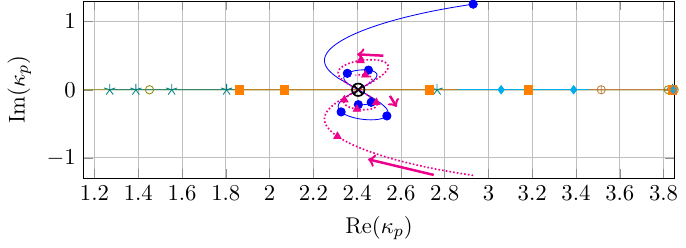}
    \caption{Seven ITE trajectories (five of which are purely real) for the unit disk in the parameter range $p\in[1.01,32]$. Arrows are used to indicate the direction of travel along the dotted magenta curve with triangular markers. Real eigenvalues travel from right to left. The symbol ``$\otimes$'' marks the position of the (Dirichlet) Laplace eigenvalue $\kappa^\star=2.4048\ldots$ (1\textsuperscript{st} zero with Bessel index $m=0$).}
    \label{fig:disk:eigs}
\end{figure}

% \begin{figure}[p]
%     \centering
%     \plotind{disk}
%     \caption{Real (left) and imaginary (center) parts of the ITE trajectories for the unit disk. The indicator $I$ is included in the right plot.}
%     \label{fig:disk:ind}
% \end{figure}

% \begin{figure}[p]
%     \centering
%     \begin{minipage}[t]{.675\linewidth}
%         \vspace{0pt}
%         \plotfitdisk
%     \end{minipage}\begin{minipage}[t]{.3\linewidth}
%         \vspace{0mm}
%         \begin{tabular}{cc}
%               & estimated \\
%              limit & rate \\
%              \hline
%              \hline
%              $p\to(p_1^\star)^-$ & 0.5845\\
%              $p\to(p_1^\star)^+$ & 0.7299\\
%              $p\to(p_2^\star)^-$ & 0.6383\\
%              $p\to(p_2^\star)^+$ & 0.6840\\
%              $p\to(p_3^\star)^-$ & 0.6607\\
%              $p\to(p_3^\star)^+$ & 0.6828
%         \end{tabular}
%         % \vspace{15mm}
%     \end{minipage}
%     \caption{Indicator $\overline{I}$ near three bifurcations for the ITE trajectories for the unit disk. Line colors and markers are the same as in \cref{fig:disk:ind}. The table displays the rate $\alpha$ such that $\overline{I}(p)\sim-|p-p^\star|^\alpha$, estimated based on the data.}
%     \label{fig:disk:fit}
% \end{figure}
\begin{figure}[t]
    \centering
    \begin{minipage}[t]{.63\linewidth}
        \vspace{0pt}
        \includegraphics[scale=.9]{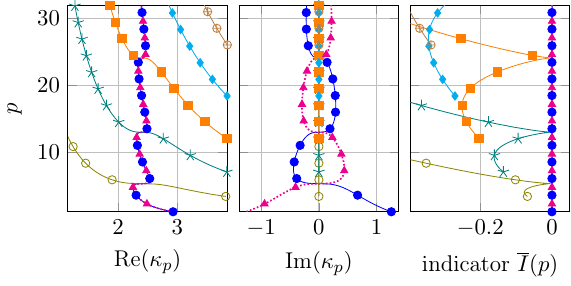}
    \end{minipage}\hspace{2mm}\begin{minipage}[t]{.35\linewidth}
        \vspace{0mm}
        \begin{tabular}{ccc}
             $p^\star$ & $\kappa_{p^\star}$ & rate $\widetilde M$ \\
             \hline
             \hline
             $(5.278)^-$ & 2.393 & 2.774\\% & 3.52
             $(5.278)^+$ & 2.393 & 2.875\\% & 1.98
             $(12.94)^-$ & 2.401 & 2.922\\% & 2.30
             $(12.94)^+$ & 2.401 & 2.909\\% & 6.65
             $(24.04)^-$ & 2.409 & 2.978\\% & 2.86
             $(24.04)^+$ & 2.409 & 3.107% & 3.68
        \end{tabular}
        % \vspace{15mm}
    \end{minipage}
    \caption{Real (left) and imaginary (center) parts of the ITE trajectories for the unit disk. The indicator $I$ is included in the right plot. The table displays the empirical bifurcation order $\widetilde M$ such that $\overline{I}(p)\sim-|\kappa_p-\kappa_{p^\star}|^{\widetilde M-1}$ as $p\to p^\star$ near each bifurcation point $p^\star$, estimated based on the data. The ITE $\kappa_{p^\star}$ differs slightly from the exact $\kappa^\star=2.4048\ldots$ because of interpolation errors.}
    \label{fig:disk:ind}
\end{figure}

In our experiments, we fix a single Bessel order $m$ for simplicity and to avoid clutter in our plots. Since eigenfunctions corresponding to different Bessel orders are orthogonal, one can simply recover the general case by superimposing the single-index MACE results for as many Bessel indices as desired.

MACE requires a non-parametric (nonlinear) eigensolver to gather each ``sample'' at a fixed value of $p$. For this, we employ the Beyn-Hankel method (specifically, Algorithm 2 in \cite{beyn2012integral} with $K=3$ Hankel blocks), a contour-integral method that is particularly suited for smooth eigenproblems like \cref{eq:disk}. As integration contour, we choose a radius-$R$ circle with some specified center, and we discretize the integral through the trapezoidal rule using $N_{\text{quad}}$ uniformly spaced quadrature points. The eigenvalue trajectories are then reconstructed using degree-$D$ B-splines. (The specific simulation parameters are provided below.) Our code is freely available at \cite{codeRepo}.

For our first test, we look at the unit-disk ITP \cref{eq:disk} for $p\in[1.01,32]$. We focus on eigenvalues with Bessel order $m=0$ and we choose a contour of radius $R=1.5$ and center $\kappa=2.5$ for the Beyn-Hankel method. The output of MACE is shown in \cref{fig:disk:eigs}.

For reference, we note that MACE achieves an overall error on the eigenvalue trajectories lower than the tolerance set at $10^{-3}$ by relying on $79$ adaptively selected collocation points. For this, $N_{\text{quad}}=5400$ quadrature points and B-splines of degree $D=7$ are used.

In this and further tests, we index eigenvalue trajectories so that the displayed curves appear smooth. Also, in all our plots, markers along curves are simply visual aids for distinguishing the different trajectories. They do \emph{not} correspond to collocation points or any other numerically relevant object. We also note that, in all our tests, parameter ranges are chosen not because of numerical or algorithmic limitations, but rather so that interesting behavior may be clearly visible in our plots without overloading them with too much information.

The results are in agreement with those of \cite[Section 3.1]{PieKle24} and \cref{prop:disk:bifurc}: cubic bifurcations arise around Bessel zeros, involving two complex-conjugate trajectories and a real one. % This emerges clearly also from the plot of real and imaginary parts of the eigenvalues in \cref{fig:disk:ind} (left and center). Therein, the bifurcation singularity appears more obviously, since some eigenvalue trajectories have large variations and unbounded derivatives.
In \cref{fig:disk:ind} (right) we display a simplified version of the indicator $I$ from \cref{eq:bifurc_indicator}, namely,
\begin{equation}\label{eq:ind_simpl}
    \overline{I}(p)=\frac{2}{a_m\pi(1-n_p)}I(p),
\end{equation}
cf.~the proof of \cref{prop:disk:bifurc}. Since the indicator depends on the specific considered eigenpair trajectory, a different curve is shown for each trajectory. We see that non-real eigenvalues have an identically zero indicator, confirming \cref{lem:zeronorm}. On the other hand, the indicator is strictly negative for real eigenvalue curves, except at bifurcation points, in agreement with \cref{cor:bifurc}. Further, as shown in the table in \cref{fig:disk:ind}, $\overline{I}$ (hence $I$) behaves approximately as $\sim-|\kappa_p-\kappa_{p^\star}|^{3-1}$ near bifurcation points, confirming the cubic bifurcation order through \cref{cor:bifurc_sufficient}.

\section{The ITP for the annulus}\label{sec:annulus}
In this section we introduce and study the ITP on the annulus with unit outer radius and inner radius $0<r<1$. Let $D=B_1\setminus\overline{B_r}$ and consider a homogeneous refractive index $n\equiv p\in\mathcal N$, cf.~\cref{eq:indexset}. The eigenfunctions are of the form
\begin{equation}\label{eq:annulus:functions}
\begin{cases}
    v(\rho,\theta)=\left(\alpha J_m(\kappa\rho)+\beta Y_m(\kappa\rho)\right)\phi(m\theta),\\
    w(\rho,\theta)=\left(\gamma J_m(\sqrt{p}\kappa\rho)+\delta Y_m(\sqrt{p}\kappa\rho)\right)\phi(m\theta),
\end{cases}\quad\rho\in[0,1],\theta\in[0,2\pi[,
\end{equation}
for arbitrary $(\alpha,\beta,\gamma,\delta)\neq(0,0,0,0)$, $m\in\mathbb{N}_0$, with $\phi$ denoting either sine (if $m\neq0$) or cosine, and with $Y_m$ being the index-$m$ Bessel function of the second kind. Enforcing the boundary conditions allows one to characterize the spectrum as
{\small%
\begin{equation}\label{eq:annulus}
    \Sigma_{\textup{annulus}(r)}(p)=\bigcup_{m=0}^\infty\left\{\kappa\in\mathbb{C}_{\neq0}:\det\begin{pmatrix}
        J_m(\kappa) & Y_m(\kappa) & J_m(\sqrt{p}\kappa) & Y_m(\sqrt{p}\kappa)\\
        J_m'(\kappa) & Y_m'(\kappa) & \sqrt{p}J_m'(\sqrt{p}\kappa) & \sqrt{p}Y_m'(\sqrt{p}\kappa)\\
        J_m(\kappa r) & Y_m(\kappa r) & J_m(\sqrt{p}\kappa r) & Y_m(\sqrt{p}\kappa r)\\
        J_m'(\kappa r) & Y_m'(\kappa r) & \sqrt{p}J_m'(\sqrt{p}\kappa r) & \sqrt{p}Y_m'(\sqrt{p}\kappa r)\\
    \end{pmatrix}=0\right\}.
\end{equation}}

It is easily seen that our results apply to this problem.

\begin{proposition}\label{prop:annulus:bifurc}
    Consider the ITP for the annulus with fixed inner radius $0<r<1$ and with parametric refractive index $n_p\equiv p$, with $\mathcal P=(p_{\min},p_{\max})\subset\mathbb R_{>0}\setminus\{1\}$, and take any continuous eigenpair trajectory, whose nonzero eigenvalue is uniformly (in $p$) semisimple except at a finite set of values of $p$.
    
    Then the eigenpair trajectory is locally analytic, with the possible exception of the values of $p$ where the eigenvalue ceases to be semisimple. Moreover, \cref{thm:bifurc}, \cref{lem:bifurc_ord}, \cref{cor:bifurc}, and \cref{cor:bifurc_sufficient} apply.
\end{proposition}
\begin{proof}
    The proof of \cref{prop:disk:bifurc}, with the obvious exception of the part pertaining to the bifurcation order, can be trivially generalized to this case.
\end{proof}

Since eigenfunction expressions are available, we can explicitly compute the indicator $I$.

\begin{proposition}\label{prop:annulus:bifurc_ord}
    Consider the setup in \cref{prop:annulus:bifurc}. Assume that the eigenpair trajectory corresponds to a Bessel index $m\in\mathbb{N}_0$, i.e., the eigenfunction trajectories are as in \cref{eq:annulus:functions}, with $\alpha$, $\beta$, $\gamma$, $\delta$, and $\kappa$ being (by \cref{thm:simple}) continuous functions of $p$. The simplified indicator \cref{eq:ind_simpl} equals $\overline{I}(p)=\|v_p\|_{L^\infty(\partial B_1;\mathbb C)}^2-r^2\|v_p\|_{L^\infty(\partial B_r;\mathbb C)}^2$.
\end{proposition}

\begin{proof}
    The proof can be found in \cref{app:bessel}.
\end{proof}

An immediate consequence is that, by \cref{cor:bifurc_sufficient}, a real (piecewise-smooth) eigenvalue trajectory undergoes bifurcation if and only if $\|v_p\|_{L^\infty(\partial B_1;\mathbb C)}=r\|v_p\|_{L^\infty(\partial B_r;\mathbb C)}$. By the assumed polar-coordinate separability in \cref{eq:annulus:functions}, this is actually equivalent to $\|v_p\|_{L^2(\partial B_1;\mathbb C)}=\|v_p\|_{L^2(\partial B_r;\mathbb C)}$.

\begin{remark}\label{rem:annulus:ord}
    Our numerical tests below indicate that bifurcations occur only for $m\neq0$, suggesting that the condition $\overline{I}(p)=0$ is impossible for any real eigenvalue curve whose Bessel index is $m=0$. Moreover, our experiments for $m\neq0$ display bifurcations only of the lowest nontrivial order $2$, suggesting that any such root of $\overline{I}$ must be of order $1/2$, by \cref{cor:bifurc_sufficient}. At the same time, we believe that higher-order bifurcations may arise in special circumstances, namely, for certain choices of $r$. However, further work is needed to confirm these statements analytically, due to the presence of the eigenfunction trajectory $v_p$, which, available only implicitly through a nonlinear eigenvalue problem, is difficult to tame.
\end{remark}

Note that, if an ITE trajectory undergoes a quadratic bifurcation on the real axis at $p=p^\star$, the trajectory must necessarily remain on the real axis over one-sided neighborhoods of $p^\star$ and leave the real axis on the opposite side of $p^\star$, cf.~\cref{rem:angle}. As discussed in \cref{rem:annulus:ord}, we have empirical evidence that this may happen for complex-conjugate eigenvalue trajectories of the ITP on the annulus. This is in contrast with the behavior displayed by the ITP on the unit disk, where non-real eigenvalue trajectories become real only ``momentarily'', at single points.

\subsection{Numerical tests on the annulus}\label{sec:annulus:num}

Here, we numerically explore the ITP on the annulus, using the same setup as in the previous section (for the unit disk). As before, our code is freely available at \cite{codeRepo}.

\subsubsection{Test with Bessel index zero}\label{sec:annulus:zero}
We look at the ITP \cref{eq:annulus} for $p\in[6,64]$ and $r=0.1$. Note the increased parameter range\footnote{We also restrict $p\geq6$ since no eigenvalues are present within the Beyn contour for $p<6$. We ignore the region $p\in[1,6]$ since an empty spectrum (i) has little numerical significance and (ii) causes numerical instabilities in our implementation of the Beyn-Hankel method. Admittedly, the latter could be easily avoided with an \emph{ad hoc} patch.} with respect to the unit disk, which also behooves us to choose a larger number of quadrature points $N_{\text{quad}}=8100$. With these changes, we aim attempt to showcase more interesting behavior in the eigenvalue trajectories.

We first focus on eigenvalues with Bessel order $m=0$, within a distance $R=1.5$ from $\kappa=3$. The output of MACE is shown in \cref{fig:annulus:eigs} (left). For reference, we note that MACE achieves a uniform error of $10^{-3}$ on the eigenvalue trajectories, using $103$ adaptively selected collocation points.

No bifurcations are observed. Instead, two complex-conjugate (non-real) eigenvalue trajectories approach the real axis as $p$ increases. This behavior is similar to that reported in \cite[Sections 3.4--3.9]{PieKle24} for other, non-ball-shaped domains. In particular, we recall that non-real eigenvalue trajectories must converge to (Dirichlet) Laplace eigenvalues of $D$ as $p\to\infty$ \cite[Theorem 2.8]{PieKle24}. In this case, the two complex-conjugate trajectories are expected to converge to $\kappa^\star$.

The plot of the real and imaginary parts of the eigenvalues in \cref{fig:annulus:ind} (left and center) provides more information. We observe that groups of four eigenvalue curves get involved in what we may describe as an ``almost-exceptional'' point. The two non-real trajectories (blue circles and magenta triangles) have rather large derivatives at the point of each of their ``orbits'' that brings them closest to the real axis. Meanwhile, two purely real trajectories (orange squares and light-blue diamonds) also exhibit very large derivatives and nearly intersect. Notably, all four above-mentioned trajectories have locally similar real parts.

In \cref{fig:annulus:ind} (right) we display the indicator $\overline{I}$ from \cref{eq:bifurc_indicator,eq:ind_simpl}, with a different curve for each trajectory. By \cref{lem:zeronorm}, only purely real trajectories yield nonzero values of $\overline{I}$. In agreement with \cref{lem:bifurc_ord}, large derivatives along real eigenvalue trajectories correspond to small values of $\overline{I}$.

Furthermore, the indicators exhibit abrupt changes in behavior near ``almost-crossing'' points. For instance, near its local maximum, the orange line with square markers has a very large derivative for $p<p^\star$ but a much smaller one for $p>p^\star$, with $p^\star$ being the corresponding ``almost-crossing'' parameter value. The light-blue line with diamond markers has a complementary behavior. This phenomenon may be categorized as ``mode veering'' in the terminology of \cite[Section 5.3]{Amsallem2011}: although the two trajectories do not intersect, a suitable perturbation of the problem would lead to intersecting trajectories.% In such a ``crossing'' case, reindexing the trajectories would be necessary to preserve smoothness: one would need to ``glue together'' the pink curve for $p<p^\star$ with the gray curve for $p>p^\star$, and vice versa.

\begin{figure}[t]
    \centering
    \includegraphics{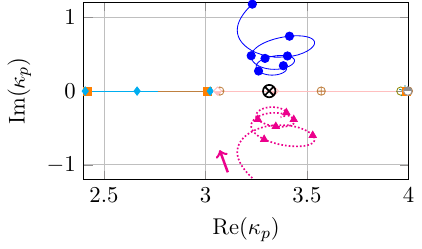}~%
    \includegraphics{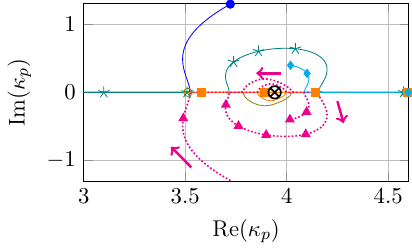}
    \caption{ITE trajectories for the annulus with $r=0.1$. Left: 10 ITEs for Bessel index $m=0$ and parameter range $p\in[6,64]$. Right: 6 ITEs for Bessel index $m=1$ and parameter range $p\in[4,25]$. The arrow is used to indicate the direction of travel along the dotted magenta curve with triangular markers. The symbols ``$\otimes$'' mark the position of the (Dirichlet) Laplace eigenvalues $\kappa^\star=3.3139\ldots$ (left, 1\textsuperscript{st} zero with Bessel index $m=0$) and $\kappa^\star=3.9409\ldots$ (right, 1\textsuperscript{st} zero with Bessel index $m=1$).}
    \label{fig:annulus:eigs}
\end{figure}
\begin{figure}[t]
    % \vspace{-5mm}
    \centering
    \includegraphics{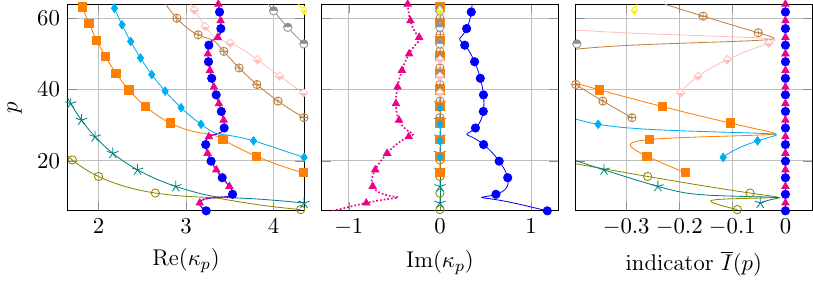}
    \caption{Real (left) and imaginary (center) parts of the ITE trajectories for the annulus with $m=0$ and $r=0.1$. The indicator $\overline{I}$ is included in the right plot.}
    \label{fig:annulus:ind}
\end{figure}

\subsubsection{Test with Bessel index one}\label{sec:annulus:one}
A careful theoretical analysis (part of a separate manuscript currently in preparation) shows that Bessel index $m=0$ is somewhat special for the ITP on the annulus. Intuitively, this is because $J_0$ is the only Bessel function of the first kind that is nonzero at the origin. Following this insight, we are behooved to repeat our numerical test for a different Bessel index $m=1$, shifting Beyn's contour to the right ($\kappa=3.5$) and focusing on the parameter range $p\in[4,25]$, in order to hone in on interesting spectral behavior. The results, shown in \cref{fig:annulus:eigs} (right), display a behavior that is different from the case $m=0$. 

\begin{figure}[t]
    \centering
    \begin{minipage}[t]{.65\linewidth}
        \vspace{0pt}
        \includegraphics[scale=.9]{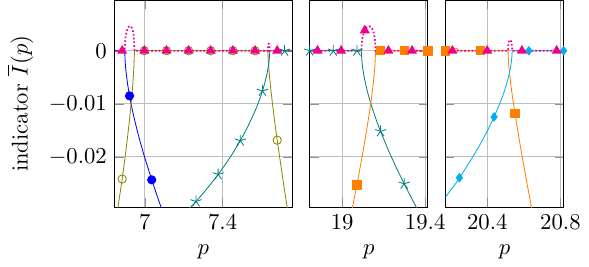}
    \end{minipage}\hspace{2mm}\begin{minipage}[t]{.33\linewidth}
        \vspace{0mm}
        \begin{small}
        \begin{tabular}{ccc}
             $p^\star$ & $\kappa_{p^\star}$ & rate $\widetilde M$ \\
             \hline
             \hline
             $(6.897)^+$ & 3.789 & 2.110\\% & 1.46
             $(6.946)^-$ & 3.525 & 1.901\\% & 3.11
             $(7.633)^+$ & 4.025 & 2.191\\% & 1.69
             $(7.641)^-$ & 4.152 & 2.284\\% & 1.91
             $(19.09)^+$ & 3.718 & 2.271\\% & 2.40
             $(19.16)^-$ & 3.855 & 2.290\\% & 2.01
             $(20.51)^+$ & 3.997 & 2.153\\% & 2.08
             $(20.54)^-$ & 4.086 & 2.154% & 2.25
        \end{tabular}
        \end{small}
        % \vspace{7mm}
    \end{minipage}
    \caption{Indicator $\overline{I}$ near eight bifurcations for the ITE trajectories for the annulus with $m=1$ and $r=0.1$. The table displays the empirical bifurcation order $\widetilde M$ such that $\overline{I}(p)\sim\pm|\kappa_p-\kappa_{p^\star}|^{\widetilde M-1}$ as $p\to p^\star$ near each bifurcation point $p^\star$, estimated based on the data.}
    \label{fig:annulus:ind2zoom}
\end{figure}

We observe eight quadratic bifurcations that turn pairs of complex-conjugate eigenvalues into pairs of real eigenvalues or vice versa. Note, in particular, how the trajectories of the eigenvalues form orbits that approach a (Dirichlet) Laplace eigenvalue $\kappa^\star$ as $p$ increases. This agrees with \cite[Theorem 2.8]{PieKle24}, which prescribes convergence of non-real trajectories to (Dirichlet) Laplace eigenvalues of $D$ as $p\to\infty$.

We show real and imaginary parts of the eigenvalues in \cref{fig:annulus:ind2} (top left and top center). The eigenvalue trajectories have extremely large derivatives near bifurcations. Notably, after complex-conjugate eigenvalues turn real, one of the ensuing real trajectories (the dotted magenta line with triangular markers) has a large \emph{positive} derivative with respect to the refractive index, which is uncharacteristic in ITPs in our experience. This trajectory travels in the ``wrong direction'' for a short time, before colliding with another real trajectory and becoming complex conjugate once again.

In \cref{fig:annulus:ind2} (top right) we display the indicator $\overline{I}$ from \cref{eq:bifurc_indicator}, with a different curve for each trajectory. As in the previous examples, only purely real trajectories yield nonzero values of $\overline{I}$. As predicted by \cref{lem:bifurc_ord}, large derivatives in real eigenvalue trajectories correspond to small values of $\overline{I}$. %Although the indicator plot may seem similar to that of the disk, cf.~\cref{fig:disk:ind}, with $\overline{I}$ taking only non-positive values, this is not the case, as the zoomed plots in \cref{fig:annulus:ind2zoom} show. Therein, we observe that the indicator becomes positive if and only if an eigenvalue trajectory is (i) real and (ii) traveling to the right. Theory confirms this, since \cref{cor:bifurc} relates the sign of $\overline{I}$ to the sign of the derivative of a real trajectory.
The zoomed plots in \cref{fig:annulus:ind2zoom} show that the indicator becomes positive if and only if an eigenvalue trajectory is (i) real and (ii) traveling to the right. Theory confirms this, since \cref{cor:bifurc} relates the sign of $\overline{I}$ to the sign of the derivative of a real trajectory.

\begin{figure}[t]
    \centering
    \includegraphics{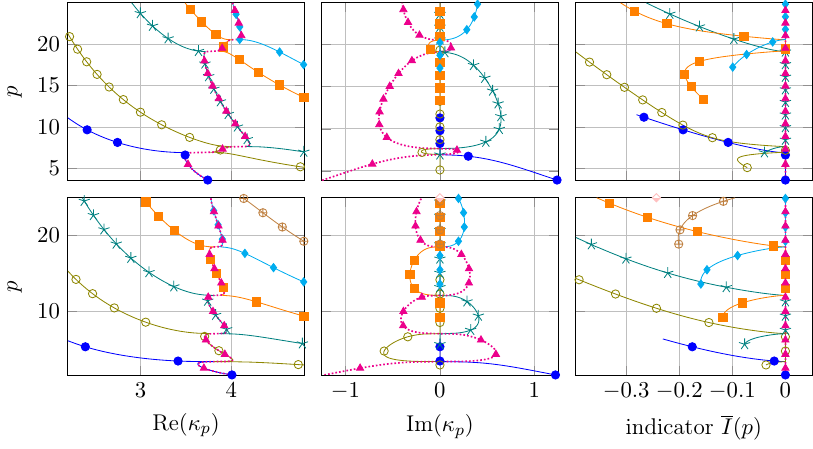}
    \caption{Real (left column) and imaginary (center column) parts of ITE trajectories, with indicator $\overline{I}$ in the right column. Top row: annulus with Bessel index $m=1$ and $r=0.1$. Bottom row: unit disk with Bessel index $m=1$.}
    \label{fig:annulus:ind2}
\end{figure}

% \begin{figure}[t]
%     \centering
%     \plotind{annulusdisk}
%     \caption{Some ITE trajectories with $m=1$ for the unit disk.}
%     \label{fig:annulus:disk}
% \end{figure}

Regarding the asymptotic behavior of $\overline{I}$ as $p\to p^\star$, in the table in \cref{fig:annulus:ind} we can observe an empirical scaling close to $\sim\pm|\kappa_p-\kappa_{p^\star}|^{2-1}$ near bifurcation points, confirming the quadratic bifurcation order by \cref{cor:bifurc_sufficient}.% In a few cases, the estimated rate is slightly higher than $2$, suggesting that bifurcations may sometimes appear as ``almost cubic''.} This is related to the fact that we are using a rather small inner radius\footnote{Upon repeating our tests with a larger $r=0.5$, we could indeed observe a scaling closer to the expected $\pm|\kappa_p-\kappa_{p^\star}|^{2-1}$. The results are omitted for conciseness.}} $r$. Indeed, our ongoing theoretical work (to appear in a separate manuscript) shows that, as $r\to0^+$, any eigenvalue of the ITP \emph{with nonzero Bessel} index converges (in some sense) to the ITP on the unit disk, cf.~\cref{sec:disk}, whose bifurcations are cubic.

Finally, we empirically compare the spectra of disk and annulus in \cref{fig:annulus:ind2} (bottom plots), where we show the results obtained by repeating our last experiment with the same setup, including the Bessel index $m=1$, but with the unit disk replacing the annulus. In the plots, we have indexed the eigenvalue trajectories in such a way as to mimic the labeling of the annulus ITEs, rather than to obtain smooth eigenvalue curves, cf.~\cref{fig:disk:ind}. We can observe a very similar structure to that in \cref{fig:annulus:eigs} (right), although now all bifurcations are cubic and the indicator $\overline{I}$ is uniformly non-positive. This seems to suggest some type of convergence of the annulus ITEs to the disk's as $r\searrow0$. As part of our ongoing work, we are theoretically investigating in which sense this convergence happens, as a simple example of ITP spectral variations under geometric perturbation.

% Intuitively, as $r\to0^+$, the roughly circular ``small'' and ``large'' orbits of eigenvalues in \cref{fig:annulus:eigs} (right) tend to become hourglass-shaped just like those in \cref{fig:disk:eigs} (left), with the real portion of all bifurcating trajectories collapsing to the bifurcation point.

\section{The inhomogeneous ITP}\label{sec:fem}
The previous two sections have considered spherically stratified media, for which the theory of complex ITEs is well developed. In cases where the ITP involves non-radially symmetric or non-homogeneous media, theoretical results are more limited. In fact, there is no theoretical guarantee of the existence or discreteness of non-real eigenvalues \cite{bookcakoni2022}.

Our results from \cref{sec:paraITP}, notably \cref{thm:bifurc_general}, apply also in the non-radially symmetric and inhomogeneous case, provided the parametrization of the refractive index satisfies the required regularity assumptions. %Unfortunately, theoretically determining the bifurcation order by inspection of the indicator $I$ through \cref{cor:bifurc_sufficient} in the general case is usually out of reach, as evidenced by the fact that cases as ``simple'' as the annulus escape this kind of analysis. On the other hand
Moreover, our computational setup makes even such problems tractable through MACE, provided a strategy for discretizing the ITP \cref{eq:ITPweak} is available. We show an example of this next.

Given the 2-dimensional unit disk $D=B_1$, we consider the inhomogeneous refractive index
\begin{equation}\label{eq:inhomogeneous}
    n_p(x,y)=p\left(1-e^{-((x-0.5)^2+y^2)/(2\sigma^2)}\right)+1,
    % n_p(x,y)=p\left(1-\text{exp}\left(-\frac1{2\sigma^2}\left(\left(x-0.5\right)^2+y^2\right)\right)\right)+1,
\end{equation}
i.e., a smooth Gaussian well ranging from a minimum value of $n_p(0.5,0)=1$ to a maximum of $n_p(-1,0)\simeq p+1$. See \cref{fig:fem:reim} (left). We fix $\sigma=0.05$, a parameter range $p\in[4,9]$, and we focus on eigenvalues within a circular contour with radius 1 and center $\kappa=5$. Note that $n_p'(x,y)=1-e^{-((x-0.5)^2+y^2)/(2\sigma^2)}>0$ on $D\setminus\{(0.5,0)\}$, so that the assumptions of \cref{thm:bifurc_general} are satisfied.

We use the finite-element (FE) method to discretize the ITP, due to its simple implementation and generality, although alternatives are available \cite{bookcakoni2022}. A main difference with respect to our previous tests is that the FE ITP discretization is a linear, non-Hermitian eigenvalue problem in $\kappa^2$, whose size depends on the resolution of the mesh used to discretize $D$. Specifically, we solve the following discrete formulation, easily derived from \cref{eq:ITPweak} \cite{Ji2013}:
\begin{equation*}
    \begin{pmatrix}
        & K_{:0}\\
        K_{:0}^\top & K_{00}
    \end{pmatrix}
    \begin{bmatrix}
        \mathbf{v}\\
        (\mathbf{w}-\mathbf{v})_0
    \end{bmatrix}=\kappa^2
    \begin{pmatrix}
        M_{n-1} & M_{n,:0}\\
        M_{n,:0}^\top & M_{n,00}
    \end{pmatrix}
    \begin{bmatrix}
        \mathbf{v}\\
        (\mathbf{w}-\mathbf{v})_0
    \end{bmatrix},
\end{equation*}
where the vectors $\mathbf{v}$ and $\mathbf{w}$ contain (complex) nodal values of the eigenfunctions $v$ and $w$, respectively, with a subscript ``$0$'' denoting removal of boundary degrees of freedom. On the other hand, $K$ and $M_\omega$ denote stiffness and mass matrices, the latter weighted by $\omega\in\{n-1,n\}$, with subscripts ``$:\hspace{-.4em}0$'' and ``$00$'' denoting removal of boundary degrees of freedom from trial functions only and from both trial and test functions, respectively. In our tests, we use a uniform mesh with resolution $h\simeq2\cdot 10^{-2}$, resulting in a discrete problem of size $7.4\cdot 10^3$.

\begin{figure}[t]
    \centering

    \begin{minipage}[t]{.325\textwidth}
        \vspace{2mm}
        \includegraphics[scale=.9]{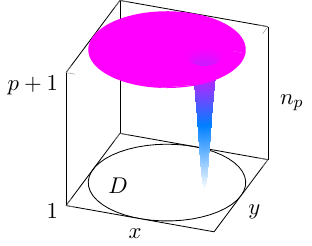}
    \end{minipage}\hspace{1mm}%
    \begin{minipage}[t]{.65\textwidth}
        \vspace{0cm}
        \includegraphics{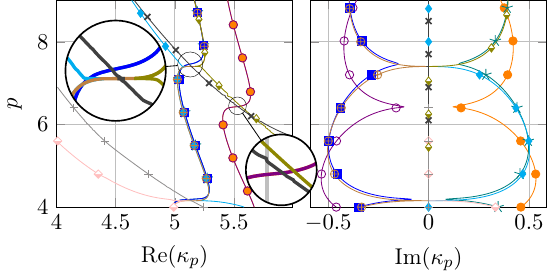}
    \end{minipage}
    \caption{Left: inhomogeneous refractive index $n_p$ from \cref{eq:inhomogeneous}. Center and right: real (center) and imaginary (right) parts of corresponding ITE trajectories (after FE discretization). Only trajectories that are non-real for some values of $p$ are plotted: 49 purely real eigenvalue trajectories are not included in this plot. Zoomed views highlight bifurcating (top left, $|\kappa_p'|\to\infty$) and almost-bifurcating (bottom right, $|\kappa_p'|\gg1$ but finite) behavior.}
    \label{fig:fem:reim}
\end{figure}

\begin{figure}[t]
    \centering

    \begin{minipage}[t]{.55\textwidth}
        \vspace{0cm}
        \includegraphics{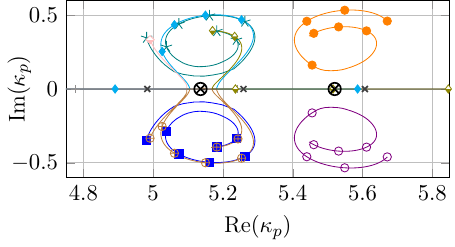}
    \end{minipage}\hspace{-2mm}%
    \begin{minipage}[t]{.4\textwidth}
        \vspace{0cm}
        \includegraphics{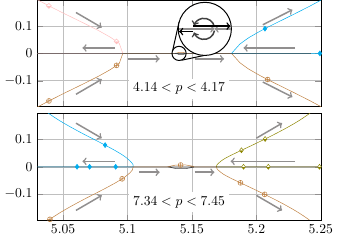}
    \end{minipage}
    \caption{Left: ITE trajectories for the (FE discretization of the) inhomogeneous ITP with \cref{eq:inhomogeneous} for the parameter range $p\in[4,9]$. The direction of travel along all non-real curves is as in \cref{fig:annulus:eigs}. 49 purely real eigenvalue trajectories are present but not visible in this plot. The symbols ``$\otimes$'' mark the positions of (Dirichlet) Laplace eigenvalues $\kappa_1^\star=5.1356\ldots$ (1\textsuperscript{st} zero with Bessel index $m=2$) and $\kappa_2^\star=5.5200\ldots$ (3\textsuperscript{rd} zero with Bessel index $m=0$). Right: zoomed views on the four major and four minor quadratic bifurcation.}
    \label{fig:fem:eigs}
\end{figure}

We use the same MACE setup as in \cref{sec:disk:num,sec:annulus:num}. The algorithm terminates after automatically identifying 55 ITE trajectories, out of which 49 are purely real. We show the results in \cref{fig:fem:reim,fig:fem:eigs}, where we still clearly see bifurcating behavior: there are four ``major'' quadratic bifurcations, visible near $\kappa=5.13\ldots$ in \cref{fig:fem:eigs} (left) and four more ``minor'' quadratic bifurcations, visible only by zooming in, cf.~\cref{fig:fem:eigs} (right). Moreover, two uniformly non-real trajectories give rise to an ``almost-exceptional'' point near $\kappa=5.52\ldots$, similar to that described in \cref{sec:annulus:zero}. As predicted by \cite[Theorem 2.8]{PieKle24}, non-real ITEs tend to be located near (Dirichlet) Laplace eigenvalues of $D$.

\begin{figure}[t]
    \centering
%    \vspace{-2mm}
    \includegraphics{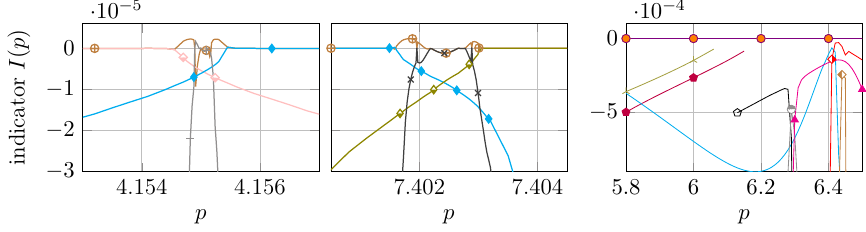}
    \caption{Trajectory indicator $I$ for the inhomogeneous ITP with \cref{eq:inhomogeneous}. Zoomed views near the first four bifurcations (left), near the last four bifurcations (center), and in a region containing %both a bifurcation happening near the integration contour and a ``nearly missed bifurcation''
    no bifurcations (right). Colors and markers are the same as in \cref{fig:fem:reim,fig:fem:eigs}. In particular, all nonzero curves in the right plot are purely real ITE trajectories, not visible in the previous plots.}
    \label{fig:fem:ind}
\end{figure}

After running MACE, we also evaluate the indicator $I$ over each ITE trajectory using definition \cref{eq:bifurc_indicator}, through (weighted) mass matrices and the discrete eigenvectors $\mathbf{v}$ and $\mathbf{w}$. This allows us to confirm whether the indicator $I$ can predict bifurcating behavior also in this inhomogeneous case. We display our results in \cref{fig:fem:ind}, where, to avoid cluttering, we restrict our focus to some specific parameter windows and on subsets of the 55 ITE trajectories. All non-real (branches of) ITE trajectories satisfy $I\equiv0$ because of \cref{lem:zeronorm}. Real ITE trajectories involved in exceptional points yield indicators $I$ that vanish exactly at the bifurcation points, as predicted by \cref{cor:bifurc}. It is again striking to see how sensitive the ITE trajectories are near bifurcations, as showcased also by the horizontal scale of the left and middle plots.

These results are not too dissimilar from those on the annulus, presented in \cref{sec:annulus:one}. In both cases, the indicator $I$ is positive for a few narrow parameter windows, corresponding to trajectories briefly moving towards the positive real axis after leaving a bifurcation point. The singular behavior here is even more complicated due to the ``minor bifurcations'', which, however, can still be well identified through the indicator $I$.% The grey indicators displayed in \cref{fig:fem:ind} (left and center) correspond to real ITE trajectories that impact other real trajectories, briefly forming a complex conjugate ITE pair. See also \cref{fig:fem:eigs} (right) for a more detailed representation of the ITE movement directions between bifurcations.

In \cref{fig:fem:ind} (right), we also show the indicator $I$ over a parameter window that does not contain any bifurcations. We do this for three reasons:
\begin{itemize}
    \item The vertical scale of this plot differs by the other two by more than an order of magnitude. This provides evidence that $I$ has a fairly large (relatively speaking) magnitude when ITE trajectories are smooth. This allows identifying exceptional points reliably.
    \item Near $p=6.4$, we see two indicator trajectories almost vanish. Similarly to the behavior observed in \cref{sec:annulus:zero}, these correspond to the almost-exceptional point involving the two uniformly non-real trajectories.
    \item We also see two indicator trajectories that seem to be heading towards zero but then disappear near $p=6.1$. These correspond to two real eigenvalue trajectories that are about to bifurcate \emph{outside the integration contour}, near the (Dirichlet) Laplace eigenvalue $\kappa^\star=3.8317\ldots<4$. The corresponding bifurcation is invisible to MACE since it is filtered out by the contour integral. Nevertheless, thanks to the indicator $I$, we can still infer its presence by its effect on the ITE trajectories within the contour.
\end{itemize}

Finally, we wish to comment on the numerical robustness of the indicator $I$. Our results show that $I$ can reliably detect spectral singularities even in large-scale (discretizations of) ITPs, although evaluating $I$ requires the ITP eigenvectors, whose computation may be unstable and inaccurate due to bifurcations. In our tests, we found some evidence of such instabilities near some of the ``minor bifurcations'', cf.~\cref{fig:fem:ind} (left and center). While future work is needed to fully analyze the potential impact of these effects (and its dependence on the ITP discretization resolution), it is crucial to note that (at least in our tests) their magnitudes are well below the MACE tolerance. As such, they do not seem to affect the predictive performance of the indicator. However, we must mention that, due to the noise in $I$, in this test we could not use a data-driven asymptotic regression $I(p)\sim|\kappa_p-\kappa_{p^\star}|^{M-1}$ to estimate the bifurcation order $M$, as done in \cref{fig:disk:ind} (right) and \cref{fig:annulus:ind2zoom} (right) in the radially symmetric cases. As a final comment regarding the use of indicator $I$ under numerical noise, we believe that any detected instability in evaluating $I$ may actually be used as supporting evidence in favor of the local presence of singularities.

\section{Conclusions and outlook}\label{sec:conclusions}
In this work, we have presented a theoretical and computational study of the parameter dependence of interior transmission eigenvalues and eigenfunctions, focusing on their smoothness with respect to variations in the refractive index. Our general results advance the understanding of ITP spectra and provide simple, practical criteria to verify the local smoothness of real ITP eigenpair trajectories. These criteria enable a quick diagnostic for identifying bifurcations and other exceptional points without resorting to costly, exhaustive numerical continuation.

By specializing our analysis to radially symmetric configurations, we have clarified the mechanisms leading to bifurcations and shown that formulating the ITP as a discrete, parametric eigenvalue problem enables the use of efficient contour-integral-based solvers. In particular, the MACE strategy proved to be a robust and accurate method for tracking eigenvalue trajectories, even through non-smooth regimes, validating our theoretical findings. In addition, our simulations were able to uncover completely novel bifurcating behavior for both non-simply connected domains and inhomogeneous media. Further efforts should be directed toward determining whether such non-smooth effects, observed in the FE-discrete version of the ITP, also arise at the continuous level.

As part of ongoing work, we are investigating extensions of our analysis to study the dependence of the ITP spectrum of the annulus on both the refractive index and the inner radius $r$. %A preliminary discussion on this and some numerical evidence have already been provided in \cref{sec:annulus:one}.
This setting also serves as a model problem for understanding more general, multi-parameter ITPs, such as the \emph{anisotropic ITP} (aITP), where parameters appear in the symmetric positive-definite matrix used to model direction-dependent properties of the underlying medium \cite{bookcakoni2022}. %This is of particular interest in inverse electromagnetic scattering \cite{bookcakoni2022}.
Under radial symmetry, the corresponding aITP can still be cast in low-dimensional form using Bessel functions, although surprisingly new spectral behavior arises \cite{lukas2020,kleefeld_computing_2018}. From a computational perspective, multi-parameter problems such as the aITP remain challenging, since most existing solvers (including MACE) are designed for a single parameter only.% Extending them to higher-dimensional parameter spaces is an open problem.

%Another important direction concerns the extension of our theoretical and computational framework to non-radially symmetric geometries, where separation of variables is no longer available and no closed-form solutions exist. In these cases, the main challenge is computational: numerical ITP solvers must eliminate spurious effects due to the essential spectrum \cite{kleefeldpieronek2018}. However, this necessarily introduces non-smooth dependence on eigenvalues or parameters, making such solvers inherently incompatible with MACE and similar tracking algorithms.% Developing new techniques that can handle such non-smooth dependencies while preserving numerical stability and spectral accuracy remains a key challenge for future research.

\appendix

\section{Proof of annulus indicator formula}\label{app:bessel}
The proof of \cref{prop:annulus:bifurc_ord} requires the following technical lemma.
\begin{lemma}\label{lem:bessel_integral}
    Given $m\in\mathbb{N}_0$, $a,b>0$ and $c,d\in\mathbb{C}$, one has
    \begin{multline}\label{eq:integral_sum}
        \int_a^bx\left|cJ_m(x)+dY_m(x)\right|^2\textup{d}x=\\
        =\left[\frac{x^2}2\left(\left|cJ_m'(x)+dY_m'(x)\right|^2+\left(1-\frac{m^2}{x^2}\right)\left|cJ_m(x)+dY_m(x)\right|^2\right)\right]_{x=a}^{x=b}.
    \end{multline}
\end{lemma}
\begin{proof}
    A proof for $d=0$ can be found, e.g., in \cite[Section 5.14]{Le70}. We prove here the general case since we could not find it elsewhere.

    Let $Z,T\in\{J_m,Y_m\}$ and $z\in\mathbb{C}_{\neq0}$. We first show that
    \begin{equation}\label{eq:integral}
        \int zZ(z)T(z)\textup{d}z=\frac{z^2}2\left(Z'(z)T'(z)+\left(1-\frac{m^2}{z^2}\right)Z(z)T(z)\right)+c.
    \end{equation}
    To this aim, we differentiate the above right-hand side and use Bessel's equation \cite[eq.~9.1.1]{abramowitz1964handbook}
    \begin{equation*}
        \zeta''(z)=\left(\frac{m^2}{z^2}-1\right)\zeta(z)-\frac1z\zeta'(z),
    \end{equation*}
    which applies to $\zeta\in\{Z,T\}$. Omitting the argument $(z)$ for conciseness, we obtain
    {\small%
    \begin{multline*}
        z\left(Z'T'+\left(1-\frac{m^2}{z^2}\right)ZT+\frac{z}2\left(Z''T'+Z'T''\right)+\frac{z}2\frac{2m^2}{z^3}ZT+\frac{z}2\left(1-\frac{m^2}{z^2}\right)\left(Z'T+ZT'\right)\right)=\\
        =z\bigg(Z'T'+\left(1-\frac{m^2}{z^2}\right)ZT+\left(\left(\frac{m^2}{2z}-\frac{z}2\right)Z-\frac12Z'\right)T'+\\
        +Z'\left(\left(\frac{m^2}{2z}-\frac{z}2\right)T-\frac12T'\right)+\frac{m^2}{z^2}ZT+\left(\frac{z}2-\frac{m^2}{2z}\right)\left(Z'T+ZT'\right)\bigg)=zZT,
    \end{multline*}}
    which verifies \cref{eq:integral}.

    Now we go back to \cref{eq:integral_sum} and expand the square. Exploiting the realness of Bessel functions for real inputs, we get
    \begin{multline*}
        \int_a^bx\left|cJ_m(x)+dY_m(x)\right|^2\textup{d}x=\\
        =|c|^2\int_a^bxJ_m(x)^2\textup{d}x+2\textup{Re}(cd)\int_a^bxJ_m(x)Y_m(x)\textup{d}x+|d|^2\int_a^bxY_m(x)^2\textup{d}x.
    \end{multline*}
    The claim follows by applying \cref{eq:integral} thrice and collapsing the square back.
\end{proof}

We can now prove \cref{prop:annulus:bifurc_ord}. Using the eigenvector formulas \cref{eq:annulus:functions}, we have
\begin{align*}
    I(p)=\int_D(|v_p|^2-n_p|w_p|^2)=\int_0^{2\pi}\phi(\theta)^2\textup{d}\theta\bigg(&\int_r^1\rho\left|\alpha_pJ_m(\kappa_p\rho)+\beta_pY_m(\kappa_p\rho)\right|^2\textup{d}\rho\\
    -&p\int_r^1\rho\left|\gamma_pJ_m(\sqrt{p}\kappa_p\rho)+\delta_pY_m(\sqrt{p}\kappa_p\rho)\right|^2\textup{d}\rho\bigg),
\end{align*}
with $\int_0^{2\pi}\phi(\theta)^2\textup{d}\theta=a_m\pi$, with $a_0=2$ and $a_m=1$ for $m>0$.

Using \cref{eq:integral_sum} and simple changes of variables we compute
{\small%
\begin{align*}
    I(p)=\frac{a_m\pi}2\bigg(&\bigg(\left|\alpha_pJ_m'(\kappa_p)+\beta_pY_m'(\kappa_p)\right|^2+\left(1-\frac{m^2}{\kappa_p^2}\right)\left|\alpha_pJ_m(\kappa_p)+\beta_pY_m(\kappa_p)\right|^2\bigg)\\[-1pt]
    -r^2&\bigg(\left|\alpha_pJ_m'(\kappa_p r)+\beta_pY_m'(\kappa_p r)\right|^2+\left(1-\frac{m^2}{\kappa_p^2r^2}\right)\left|\alpha_pJ_m(\kappa_p r)+\beta_pY_m(\kappa_p r)\right|^2\bigg)\\[-1pt]
    -p&\bigg(\left|\gamma_pJ_m'(\sqrt{p}\kappa_p)+\delta_pY_m'(\sqrt{p}\kappa_p)\right|^2+\left(1-\frac{m^2}{p\kappa_p^2}\right)\left|\gamma_pJ_m(\sqrt{p}\kappa_p)+\delta_pY_m(\sqrt{p}\kappa_p)\right|^2\bigg)\\[-1pt]
    +pr^2&\bigg(\left|\gamma_pJ_m'(\sqrt{p}\kappa_p r)+\delta_pY_m'(\sqrt{p}\kappa_p r)\right|^2\\
    &\hspace{-1mm}+\left(1-\frac{m^2}{p\kappa_p^2r^2}\right)\left|\gamma_pJ_m(\sqrt{p}\kappa_p r)+\delta_pY_m(\sqrt{p}\kappa_p r)\right|^2\bigg)\bigg).
\end{align*}}

By the boundary conditions at $\rho\in\{r,1\}$,
\begin{align*}
    \alpha_pJ_m(\kappa_p)+\beta_pY_m(\kappa_p)=&~\gamma_pJ_m(p\kappa_p)+\delta_pY_m(p\kappa_p)\\
    \alpha_pJ_m'(\kappa_p)+\beta_pY_m'(\kappa_p)=&~\sqrt{p}\left(\gamma_pJ_m'(\sqrt{p}\kappa_p)+\delta_pY_m'(\sqrt{p}\kappa_p)\right)\\
    \alpha_pJ_m(\kappa_p r)+\beta_pY_m(\kappa_p r)=&~\gamma_pJ_m(p\kappa_p r)+\delta_pY_m(p\kappa_p r)\\
    \alpha_pJ_m'(\kappa_p r)+\beta_pY_m'(\kappa_p r)=&~\sqrt{p}\left(\gamma_pJ_m'(\sqrt{p}\kappa_p r)+\delta_pY_m'(\sqrt{p}\kappa_p r)\right).
\end{align*}
This leads to several cancellations, ultimately yielding
\begin{equation*}
    I(p)=\frac{a_m\pi}2\left(1-p\right)\left(\left|\alpha_pJ_m(\kappa_p)+\beta_pY_m(\kappa_p)\right|^2-r^2\left|\alpha_pJ_m(\kappa_p r)+\beta_pY_m(\kappa_p r)\right|^2\right).
\end{equation*}
The claim follows by \cref{eq:annulus:functions}.

\section*{Acknowledgments}
Our thanks go to the anonymous reviewers, whose constructive and detailed comments helped us in improving this paper.

\bibliographystyle{abbrv}
\bibliography{references}

\end{document}